\documentclass[12pt]{article}
\usepackage{amssymb,amsmath,bm,amsthm,dsfont}
\DeclareMathOperator*{\esssup}{\rm ess\,\sup}
\usepackage{mathrsfs}
\usepackage{hyperref}
\usepackage{graphicx}
\hypersetup{
    colorlinks,
    citecolor=black,
    filecolor=black,
    linkcolor=black,
    urlcolor=black
}
\usepackage{enumerate}
\begin{document}
\newcommand{\beq}{\begin{eqnarray}}
\newcommand{\eeq}{\end{eqnarray}}
\newcommand{\beas}{\begin{eqnarray*}}
\newcommand{\enas}{\end{eqnarray*}}
\newcommand{\bea}{\begin{eqnarray}}
\newcommand{\ena}{\end{eqnarray}}
\newcommand{\bms}{\begin{multline*}}
\newcommand{\ems}{\end{multline*}}
\newcommand{\qmq}[1]{\quad \mbox{#1} \quad}
\newcommand{\qm}[1]{\quad \mbox{#1}}
\newcommand{\nn}{\nonumber}
\newcommand{\bbox}{\hfill $\Box$}
\newcommand{\ignore}[1]{}
\newcommand{\tr}{\mbox{tr}}
\newcommand{\Bvert}{\left\vert\vphantom{\frac{1}{1}}\right.}
\newtheorem{theorem}{Theorem}[section]
\newtheorem{corollary}{Corollary}[section]
\newtheorem{conjecture}{Conjecture}[section]
\newtheorem{proposition}{Proposition}[section]
\newtheorem{remark}{Remark}[section]
\newtheorem{lemma}[theorem]{Lemma}
\newtheorem{example}{Example}[section]
\newtheorem{definition}{Definition}[section]
\newtheorem{condition}{Condition}[section]
\newcommand{\pf}{\noindent {\bf Proof:} }
\newcommand{\sbull}{\scalebox{0.5}{\textbullet}}
\newcommand{\LargerCdot}{\raisebox{-0.25ex}{\scalebox{1.2}{$\cdot$}}}

\title{{\bf\Large Dickman approximation in simulation, summations and perpetuities}}

\author{Chinmoy Bhattacharjee, Larry Goldstein\thanks{This work was partially supported by NSA grant
		H98230-15-1-0250.}}
	
	\newcommand{\Addressc}{{
			\bigskip
			\footnotesize
			
 \textsc{IMSV, Universit{\"a}t Bern, Switzerland}\par\nopagebreak
 \textit{E-mail address}: \texttt{ \href{mailto:chinmoy.bhattacharjee@stat.unibe.ch}{chinmoy.bhattacharjee@stat.unibe.ch}}
}}

\newcommand{\Addressl}{{
		\bigskip
		\footnotesize
		
		\textsc{Department of Mathematics, University of Southern California, Los Angeles, USA}\par\nopagebreak
		\textit{E-mail address}: \texttt{ \href{mailto:larry@math.usc.edu}{larry@math.usc.edu}}
}}


\maketitle

\begin{abstract}
	The generalized Dickman distribution ${\cal D}_\theta$ with parameter $\theta>0$ is the unique solution to the distributional equality 
	$W=_d W^*$, where
	\bea \label{eq:W*.trans}
	W^*=_d U^{1/\theta}(W+1),
	\ena
	with $W$ non-negative with probability one, $U \sim {\cal U}[0,1]$ independent of $W$, and $=_d$ denoting equality in distribution. Members of this family appear in number theory, stochastic geometry, perpetuities and the study of algorithms. We obtain bounds in Wasserstein type distances between ${\cal D}_\theta$ and the distribution of
	\bea \label{eq:rand.sum}
	W_n= \frac{1}{n} \sum_{i=1}^n  Y_k B_k 
	\ena
	where $B_1,\ldots,B_n, Y_1, \ldots, Y_n$ are independent with $B_k \sim {\rm Ber}(1/k), E[Y_k]=k, {\rm Var}(Y_k)=\sigma_k^2$ and provide an application to the minimal directed spanning tree in $\mathbb{R}^2$, and also obtain such bounds when the Bernoulli variables in \eqref{eq:rand.sum} are replaced by Poissons. We also give simple proofs and provide bounds with optimal rates for the Dickman convergence of the weighted sums, arising in probabilistic number theory, of the form
		\beas
		S_n=\frac{1}{\log(p_n)} \sum_{k=1}^n X_k \log(p_k) 
		\enas
		where $(p_k)_{k \ge 1}$ is an enumeration of the prime numbers in increasing order and $X_k$ is geometric with parameter $(1-1/p_k)$, Bernoulli with success probability $1/(1+p_k)$ or Poisson with mean $\lambda_k$.
	
	In addition, we broaden the class of generalized Dickman distributions by studying the fixed points of the transformation
	\beas
	s(W^*)=_d U^{1/\theta}s(W+1)
	\enas
	generalizing \eqref{eq:W*.trans}, that allows the use of non-identity utility functions $s(\cdot)$ in Vervaat perpetuities. We obtain distributional bounds for recursive methods that can be used to simulate from this family.
\end{abstract}

\section{Introduction} 
The Dickman distribution ${\cal D}$ first made its appearance in \cite{Di30} in the context of number theory for counting the number of integers below a fixed threshold whose prime factors lie below a given upper bound; see the more recent work \cite{Pi16} for a readable explanation of how the Dickman distribution arises there. Members from the broader class of generalized Dickman distributions ${\cal D}_\theta$ for $\theta>0$, of which ${\cal D}={\cal D}_1$, have since been used to approximate counts in 
logarithmic combinatorial structures, including permutations and partitions in \cite{ABT}, and more generally for the quasi-logarithmic class considered in \cite{BaNi11}, for the weighted sum of edges connecting vertices to the origin in minimal directed spanning trees in \cite{PW04}, and for certain weighted sums of independent random variables in \cite{Pi16b}. Simulation of the generalized Dickman distribution has been considered in \cite{Dev}, and in connection with the Quickselect sorting algorithm in \cite{Hwa} and \cite{Go17v1}.

Following \cite{Go17v1}, for a given $\theta>0$ and non-negative random variable $W$, define the $\theta$-Dickman bias distribution of $W$ by
\bea \label{def:W*}
W^*=_d U^{1/\theta}(W+1),  
\ena
where $U \sim {\cal U}[0,1]$ and is independent of $W$, and $=_d$ denotes equality in distribution. 
Though the density of ${\cal D}_\theta$ can presently be given only by specifying it somewhat indirectly as a certain solution to a differential delay equation, it is well known \cite{Dev} that the distributions ${\cal D}_\theta$ are characterized by satisfying $W^*=_d W$ uniquely, that is, ${\cal D}_\theta$ is the unique fixed point of the distributional transformation \eqref{def:W*}. Indeed, this property is the basis for simulating from this family using the recursion
\bea \label{eq:sim.dickman.theta}
W_{n+1}=U_n^{1/\theta}(W_n+1) \qmq{for $n \ge 0$, with $W_0=0$,}
\ena
where $U_m, m \ge 0$ are i.i.d. ${\cal U}[0,1]$ random variables and $U_n$ is independent of $W_n$, see \cite{Dev}.

Generally, distributional characterization and their associated transformations, such as \eqref{def:W*}, provide an additional avenue to study distributions and their approximation, and have been considered for the normal \cite{GoReG97}, the exponential \cite{PR11}, and various other distributions that may be less well known, such as one arising in the study of the degrees of vertices in certain preferential attachment graphs, see \cite{PRR13}.

In the following, $D_\theta$ will denote a ${\cal D}_\theta$ distributed random variable, where the subscript may be dropped when equal to 1. In \cite{Go17v2}, the upper bound 
\bea \label{eq:bound17}
d_{1}(W,D_\theta) \le (1+\theta)d_1(W,W^*)
\ena
for the Wasserstein distance between a non-negative random variable $W$ and $D_\theta$ 
was proved via Stein's method, where
\bea \label{def:d1}
d_1(X,Y) = \sup_{h \in {\rm Lip}_1}|Eh(X)-Eh(Y)|
\ena
with
\bea \label{def:Lipalpha}
{\rm Lip}_\alpha = \{h: |h(x)-h(y)| \le \alpha |x-y|\} \qm{for $\alpha \ge 0$.}
\ena
We also apply the fact that
alternatively one can write
\bea \label{def:was.inf}
d_1(X,Y)=\inf E|X-Y|,
\ena
where the infimum is over all joint distributions having the given $X,Y$ marginals. The infimum is achieved for variables taking values in any Polish space, see e.g. \cite{Ra91}, and so in particular for those that are real valued.
For notational simplicity we write $d_1(X,Y)$, say, for 
$d_1({\cal L}(X),{\cal L}(Y))$, where ${\cal L}(\cdot)$ stands for the distribution, or law, of a random variable.  In \cite{Go17v2},  inequality \eqref{eq:bound17} was used to derive a bound on the quality of the Dickman approximation for the running time of the Quickselect algorithm.

Here our aim is two fold. First, in Section \ref{sec:sums} we study the approximation of sums that converge in distribution to Dickman, for instance, those of the form
\bea \label{eq:weigh.sum.Y}
W_n = \frac{1}{n}\sum_{k=1}^n Y_kB_k, 
\ena
where $\{B_1,\ldots,B_n,Y_1,\ldots,Y_n\}$ are independent, $B_k$ is a Bernoulli random variable with success probability $1/k$, and $Y_k$ is non-negative with $EY_k=k$, and ${\rm Var}(Y_k)=\sigma_k^2$ for all $k=1,\ldots,n$. The most well known case is the one where $Y_k=k$ a.s., 
for which
\bea \label{def:weighted.sums.Bernoullis}
W_n = \frac{1}{n}\sum_{k=1}^n kB_k.
\ena
Sums of this type arise, for instance, in the analysis of the Quickselect algorithm for finding the $m^{th}$ smallest of a list of $n$ distinct numbers, see \cite{Hwa} (also \cite{Go17v2}), and for the sum of positions of records in a uniformly random permutation (see \cite{Re62}).
To state the result we will apply to such sums, we first define the Wasserstein-2 metric
\bea\label{def:d111}
d_{1,1}(X,Y) = \sup_{h \in {\cal H}_{1,1}}|Eh(Y)-Eh(X)| 
\ena
where, for $\alpha \ge 0,\beta \ge 0$,
\bea \label{eq:def.calHab}
{\cal H}_{\alpha,\beta} = \{h: h \in {\rm Lip}_\alpha, h' \in {\rm Lip}_\beta\},
\ena
with ${\rm Lip}_\alpha$ given in \eqref{def:Lipalpha}.
The work \cite{AMPS} obtains a bound of the form $C \sqrt{\log n}/n$ between $W_n$ in \eqref{def:weighted.sums.Bernoullis} and $D$ in a metric weaker than $d_{1,1}$ in \eqref{def:d111},  requiring test functions to be three times differentiable, and with the constant $C$ unspecified. The following theorem provides a more general result that in the specific case of \eqref{def:weighted.sums.Bernoullis} yields a bound in the stronger metric $d_{1,1}$ with a small, explicit constant.
\begin{theorem} \label{thm:weighted.Bernoullis}
	Let $W_n$ be as in \eqref{eq:weigh.sum.Y} and $D$ a standard Dickman random variable. Then with the metric $d_{1,1}$  in \eqref{def:d111},
	\beas
	d_{1,1}(W_n,D) \le \frac{3}{4n}+ \frac{1}{2n^2} \sum_{k=1}^n \frac{1}{k} \sqrt{(\sigma_k^2 + k^2) \sigma_k^2},
	\enas
	and in particular if $Y_k=k$ a.s., that is, for $W_n$ as in \eqref{def:weighted.sums.Bernoullis},
	\bea \label{eq:ref.for.rec}
	d_{1,1}(W_n,D) \le \frac{3}{4n}.
	\ena
\end{theorem}
From the first bound given by the theorem, speaking asymptotically we see that $W_n$ in \eqref{eq:weigh.sum.Y} converges to $D$ in distribution whenever $\sum_{k=1}^n \frac{1}{k} \sqrt{(\sigma_k^2 + k^2) \sigma_k^2}=o(n^2)$. In particular, weak convergence to the Dickman distribution occurs if $\sigma_k^2=O(k^{2-\epsilon})$ for some $\epsilon>0$. In Section \ref{sec:sums} we provide an application of Theorem \ref{thm:weighted.Bernoullis} to minimal directed spanning trees in $\mathbb{R}^2$.

We also show the following related result for a weighted sum of independent Poisson variables. For $\lambda > 0$, let ${\cal P}(\lambda)$ denote a Poisson random variable with mean $\lambda$.
\begin{theorem} \label{thm:Poisson}
	For $\theta > 0$, let $\{P_1,\ldots,P_n,Y_1,\ldots,Y_n\}$ be independent with $P_k \sim {\cal P}(\theta/k)$ and $Y_k$ non-negative with $EY_k=k$ and ${\rm Var}(Y_k)=\sigma_k^2$, for all $k=1,\ldots,n$. Then 
	\bea \label{eq:Wn.YkPk}
	W_n=\frac{1}{n} \sum_{k=1}^n Y_k P_k
	\ena
	satisfies
	\bea\label{PoisY}
	d_{1,1}(W_n,D_\theta) \le \frac{\theta}{4n} +\frac{\theta}{n} \sum_{k=1}^n \frac{\sigma_k}{k}+ \frac{\theta}{2n^2} \sum_{k=1}^n \frac{1}{k} \sqrt{(\sigma_k^2 + k^2) \sigma_k^2},
	\ena
	and in particular, in the case $Y_k=k$ a.s.,
	\bea\label{Pois}
	W_n=\frac{1}{n} \sum_{k=1}^n k P_k \qmq{satisfies}d_{1,1}(W_n,D_\theta) \le \frac{\theta}{4n}.
	\ena
\end{theorem}
Similar to the weighted sum of Bernoullis in \eqref{eq:weigh.sum.Y}, we have weak convergence to the Dickman distribution if $\sigma_k^2=O(k^{2-\epsilon})$ for some $\epsilon>0$.

Next, we study Dickman approximation of weighted geometric and Bernoulli sums that appear in probabilistic number theory. For geometric variables, we write $X \sim {\rm Geom}(p)$ if $P(X=m)=(1-p)^m p$ for $m \ge 0$. Let $(p_k)_{k \ge 1}$ be an enumeration of the prime numbers in increasing order and $\Omega_{n}$ denote the set of all positive integers having no prime factor larger than $p_n$. Let $X_1,\ldots,X_n$ be independent with $X_k \sim {\rm Geom}(1-1/p_k)$ for $1 \le k \le n$, and let $\Pi_n$ be the distribution of $M_n$ given by
	\bea \label{eq:sum.pnt}
	M_{n}=\prod_{k=1}^n p_k^{X_k} \qmq{and}
	S_n=\frac{\log M_{n}}{\log (p_n)}= \frac{1}{\log(p_n)} \sum_{k=1}^n X_k \log(p_k).
	\ena
	One can specify (see e.g. \cite{Pi16}) $\Pi_n$ by
	\beas
	\Pi_n(m)=\frac{1}{\pi_nm}  \qmq{for $m \in \Omega_n$}
	\enas 
	with normalizing constant necessarily satisfying $\pi_n=\sum_{m \in \Omega_n} 1/m$. Distributional convergence of $S_n$ to the standard Dickman distribution was proved in \cite{Pi16}. In Theorem \ref{thm:num1} below, we provide $(\log n)^{-1}$ convergence rate in the Wasserstein-2 norm.
	\begin{theorem}\label{thm:num1}
		For $D$ a standard Dickman random variable and $S_n$ as in \eqref{eq:sum.pnt} with $X_1,\ldots,X_n$ independent variables with $X_k \sim {\rm Geom}(1-1/p_k)$, we have
		\beas
		d_{1,1}(S_n,D) \le \frac{C}{\log n}
		\enas
		for some universal constant $C$. Moreover, the order is not improvable.
	\end{theorem}
One may instead consider the distribution $\Pi'_n$ over $\Omega'_n$, the set of square-free integers with largest prime factor less than or equal to $p_n$, with $\Pi'_n(m)$ proportional to $1/m$ for all $m \in \Omega'_n$. Then $M_{n}= \prod_{k=1}^n p_k^{X_k}$ has distribution $\Pi'_n$ when 
$X_k \sim {\rm Ber}(1/(1+p_k))$ and are independent (see e.g. \cite{CS13}). That $S_n=\log M_{n}/\log (p_n)$ converges in distribution to the standard Dickman was proved in \cite{CS13} and very recently a $(\log \log n)^{3/2} (\log n)^{-1}$ rate was provided in \cite{AMPS} in a metric defined as a supremum over a class of three times differentiable functions. We provide the improved $(\log n)^{-1}$ convergence rate in the stronger Wasserstein-2 norm.
	\begin{theorem}\label{thm:num}
	For $D$ a standard Dickman random variable and $S_n$ as in \eqref{eq:sum.pnt} with $X_1,\ldots,X_n$ independent variables with $X_k \sim {\rm Ber}(1/(1+p_k))$, we have
	\beas
	d_{1,1}(S_n,D) \le \frac{C}{\log n}
	\enas
	for some universal constant $C$. Moreover, the order is not improvable.
\end{theorem}
In Examples \ref{ex:Poisson1} and \ref{ex:Poisson2} we also provide such bounds when the $X_k$'s in \eqref{eq:sum.pnt} are distributed as Poisson random variables with parameters $\lambda_k>0$ given by certain functions of $p_k$. For our results in probabilistic number theory, we closely follow the arguments in \cite{AMPS}.

In Section \ref{sec:perp} we consider the connection between the class of Dickman distributions and perpetuities. By approaching from the view of utility, we extend the scope of the Dickman distributions past the currently known class. The recursion \eqref{eq:sim.dickman.theta} was interpreted by Vervaat, see \cite{Ve79}, as the relation between the values of a perpetuity at two successive times. In particular, during the $n^{th}$ time period a deposit of some fixed value, scaled to be unity, is added to the value of an asset. During that time period, a multiplicative factor in $[0,1]$, accounting for depreciation is applied; in \eqref{eq:sim.dickman.theta} that factor is taken to be
$U^{1/\theta}$. The generalized Dickman distributions arise as fixed points of this recursion, that is, solutions to $W^*=_dW$ where $W^*$ is given in \eqref{def:W*}.

Measuring the value of an asset directly by its monetary value corresponds to the case where the utility function $s(\cdot)$ of an asset is taken to be the identity. We consider the generalization of \eqref{eq:sim.dickman.theta} to
\bea \label{eq:tVervaat}
s(W_{n+1})=U_n^{1/\theta}s(W_n+1).
\ena
In \cite{Be38}, see also the translation \cite{Be54}, Daniel Bernoulli argued that utility should be given as a concave function of the value of an asset, typically justified by observing that receiving one unit of currency would be of more value to an individual who has very few resources than one who has resources in abundance, see \cite{ELS05}. We may then interpret \eqref{eq:tVervaat} in a manner similar to \eqref{eq:sim.dickman.theta}, but now in terms of utility. Again, during the $n^{th}$ time period, a constant value, scaled to be one, is added to an asset. Then, at time $n+1$, the utility of the asset is given by some discount factor applied to the incremented utility of the asset. When $s(\cdot)$ is invertible, as for the most common Vervaat perpetuities,  one can now gain insight into their long term behavior by studying fixed points of the transformation
\bea\label{bias}
W^*=_d s^{-1}(U^{1/\theta} s(W+1)).
\ena

Theorem \ref{existence} in Section \ref{sec:perp} shows that under mild and natural conditions on the utility function $s(\cdot)$ the transformation \eqref{bias} has a unique fixed point, say $D_{\theta,s}$, which we say has the $(\theta,s)$-Dickman distribution, denoted here as ${\cal D}_{\theta,s}$. As the identity function $s(x)=x$ recovers the class of generalized Dickman distributions, this extended class strictly contains them. The  parameter $\theta>0$ here plays the same role for ${\cal D}_{\theta, s}$ as it does for ${\cal D}_{\theta}$, in particular in its appearance in the distributional bounds for simulation using recursive schemes. Theorem \ref{bound} generalizes the bound \eqref{eq:bound17} of \cite{Go17v2} to the ${\cal D}_{\theta,s}$ family, providing the inequality
\bea \label{eq:d1bd.theta.s}
d_1(W,D_{\theta,s}) \le (1-\rho)^{-1} d_1(W^*,W)
\ena
with a parameter $\rho$ given by a bound on an integral involving $\theta$ and $s(\cdot)$, see \eqref{I:def} and \eqref{rholess1}.

We apply \eqref{eq:d1bd.theta.s} to assess the quality of the recursive scheme 
\bea \label{eq:intro.Wn.iteration}
W_{n+1}= s^{-1}(U_n^{1/\theta} s(W_n+1)) \qm{for $n \ge 0$ and $W_0=0$,}
\ena
for the simulation of variables having the ${\cal D}_{\theta,s}$ distribution. Simulation by these means for the ${\cal D}_\theta$ family was considered in \cite{Dev}, though no bounds on its accuracy were provided. An algorithmic method for the exact simulation from the ${\cal D}_\theta$ family was given in \cite{FiHu10} with bounds on the expected running time. In brief, the method in \cite{FiHu10} depends on the use of a multigamma coupler as an update function for the kernel $K(x,\cdot):=\mathcal{L}(U^{1/\theta}(x+1))$, and on finding a dominating chain so that one can simulate from its stationary distribution, a shifted geometric distribution in this case. To extend this approach to the more general family ${\cal D}_{\theta,s}$, one would consider the kernel $K(x,\cdot):=\mathcal{L}(U^{1/\theta}s(x+1))$, and though one can generalize the multigamma coupler for use as an update function for this kernel, finding a suitable dominating chain in this generality may not be straightforward.

The efficacy of a simpler recursive scheme for simulation from this family is addressed in \eqref{GenSimBd} of Corollary \ref{othertwo}  where we show that the iterates generated by \eqref{eq:intro.Wn.iteration} obey the inequality
\beas
d_1(W_n, D_{\theta,s})
\le(1-\rho)^{-1} \left(\frac{\theta}{\theta +1}\right)^n E[s^{-1}(U^{1/\theta})],
\enas
and which thus exhibit exponentially fast convergence.
In Section \ref{subsec:dickman.examples} we present some instances from the family ${\cal D}_{\theta,s}$ that arise as limiting distributions for perpetuities when taking our utilities $s(\cdot)$ from those studied in economics.

We obtain our results by extensions of \cite{Go17v1} for the Stein's method framework for the Dickman distribution. The application of Stein's method, as unveiled in \cite{St72} and further developed in \cite{St86}, begins with a characterizing equation for a given target distribution. Such a characterization is then used as the basis to form a Stein equation, which is usually a difference or differential equation involving test functions in a class corresponding to a desired probability metric, such as the class of ${\rm Lip}_1$ functions for the Wasserstein distance in \eqref{def:d1}. One key step of the method requires bounds on the smoothness of solutions over the given class of test functions. For a modern treatment of Stein's method, see \cite{CGS} and \cite{Ro11}.

Theorems \ref{thm:num} improves on results of \cite{AMPS}. That work applies a different version of Stein's method, and in particular does not consider any form of the Stein equation, such as \eqref{eq:Stein.smoothed.form} or \eqref{eq:stein.eq.gen.dickman}. Consequently \cite{AMPS} does not obtain bounds on a Stein solution for any Dickman case, as is achieved here in Theorems \ref{Lipschitzg} and \ref{fsolutionbds}.
Indeed, there it is noted in \cite{AMPS1} that this last step can be an `extremely difficult problem'.

In \cite{Go17v1} the Stein equation used for the ${\cal D}_\theta$ family was of the integral type
\bea \label{eq:Stein.smoothed.form}
g(x)-A_{x+1}g=h(x)-E[h(D_{\theta})]
\ena
where the averaging operator $A_x g$ was given by
\beas
A_x g = \left\{
\begin{array}{cc}
	g(0) & \mbox{for $x=0$}\\
	\frac{\theta}{x^\theta} \int_0^x  g(u)u^{\theta-1}du & \mbox{for $x>0$}.
\end{array}
\right.
\enas
To handle the ${\cal D}_{\theta,s}$ family, over the range $x>0$ we generalize the form of the averaging operator to
\bea \label{eq:averaging.tx}
A_xg=\frac{1}{t(x)}\int_0^x g(u)t'(u)du,
\ena
where $t(x)=s^\theta(x)$. Smoothness bounds for solutions of \eqref{eq:Stein.smoothed.form} with $A_x$ as in \eqref{eq:averaging.tx} and $D_\theta$ replaced by $D_{\theta,s}$, are given in Theorem \ref{Lipschitzg} in Section \ref{Smoothness} for a wide range of functions $s(\cdot)$. This generalization requires significant extensions of existing methods.

Use of the Stein equation \eqref{eq:Stein.smoothed.form} is appropriate when the variable $W$ of interest can be coupled to some $W^*$ with its $\theta$-Dickman bias distribution. However, such direct couplings appear elusive for all our examples in Section \ref{sec:sums}, including in particular those in probabilistic number theory, and a different approach is needed. To handle these new examples we consider instead a new Stein equation, of differential-delay type, given by
\bea \label{eq:stein.eq.gen.dickman}
(x/\theta) f'(x)+f(x)-f(x+1)=h(x)-E[h(D_{\theta})].
\ena
To apply the method, uniform bounds on the smoothness of the solution $f(\cdot)$ over test functions $h(\cdot)$ in some class ${\cal H}$ is required; we achieve such bounds for the class ${\cal H}_{1,1}$ in Theorem \ref{fsolutionbds} in Section \ref{Smoothness}.

Throughout the paper, for a real-valued measurable function $f(\cdot)$ on a domain $S \subset \mathbb{R}$, $\|f\|_\infty$ denotes its essential supremum norm defined by
\bea\label{esssup:def}
\|f\|_\infty =\esssup_{x \in S} |f(x)|= \inf \{b \in \mathbb{R}: m(\{x:f(x)>b\})=0\},
\ena
where $m$ denotes the Lebesgue measure on $\mathbb{R}$. For any real valued function defined on $A \subset S$ we define its supremum norm on $A$ by
	\bea\label{supnorm}
	\|f\|_{A}= \sup_{x \in A}|f(x)|.
	\ena
Unless otherwise specifically noted, integration will be with respect to $m$, which for simplicity will be denoted by, say, $dv$ when the variable of integration is $v$.

This work is organized as follows. We focus on sums, such as the Bernoulli and Poisson weighted sums in \eqref{eq:weigh.sum.Y} and \eqref{eq:Wn.YkPk}, and sums arising in probabilistic number theory as \eqref{eq:sum.pnt}, in Section \ref{sec:sums}. We focus on perpetuities, with examples, in Section \ref{sec:perp}, and in Section \ref{Smoothness} we prove smoothness bounds on the two types of Stein solutions considered here.

\section{Dickman Approximation of Sums}\label{sec:sums}
We will prove Theorems \ref{thm:weighted.Bernoullis} and \ref{thm:Poisson}, starting with a simple application of the former, in Section \ref{sub:sums}, and then provide the proofs of Theorems \ref{thm:num1} and \ref{thm:num}, in probabilistic number theory, in Section \ref{sub:pnt}. In this section we deal with the form \eqref{eq:stein.eq.gen.dickman} of the Stein equation. That is, in the proofs of Theorems \ref{thm:weighted.Bernoullis},  \ref{thm:Poisson} and \ref{thm:numgen}, we take a fixed $\theta>0$ and $h \in\mathcal{H}_{1,1}$, the
function class defined in \eqref{eq:def.calHab}, and let $f \in
\mathcal{H}_{\theta,\theta/2}$ be the solution of the Stein equation \eqref
{eq:stein.eq.gen.dickman} that is guaranteed by Theorem \ref
{fsolutionbds}. Substituting our $W_n$ of interest for $x$ in \eqref{eq:stein.eq.gen.dickman} and taking expectation yields
\bea \label{eq:SteinEqWn}
E[h(W_n)]-E[h(D_\theta)]=E\left[ (W_n/\theta) f'(W_n)-(f(W_n+1)-f(W_n)) \right].
\ena

\subsection{Weighted Bernoulli and Poisson Sums}\label{sub:sums}

We begin with a simple application of Theorem \ref{thm:weighted.Bernoullis} to the minimal directed spanning tree, or MDST, following \cite{BhRo}, first pausing to describe the construction of the MDST.

For two points $(u_1,v_1)$ and $(u_2,v_2)$ in $\mathbb{R}^2$, we write $(u_1,v_1) \preceq (u_2,v_2)$ if $u_1 \le u_2$ and $v_1 \le v_2$, and write $(u_1,v_1) \not \preceq (u_2,v_2)$ otherwise. For any set of points $\cal V$ in $\mathbb{R}^2$, we say $(u,v) \in {\cal V}$ is a minimal point, or sink, of $\cal V$ if $(a,b) \not \preceq (u,v)$ for all $(a,b) \in {\cal V}, (a,b)\not =(u,v)$.

For $n \in \mathbb{N}$, consider a set of $n+1$ distinct points ${\cal V}=\{ (a_i,b_i), 0 \le i \le n \}$ in $[0,1] \times [0,1]$ where we take $(a_0,b_0)=(0,0)$, the origin. Let $E$ be the set of directed edges $(a_i,b_i) \to (a_j,b_j)$ with $i \not= j$ and $(a_i,b_i) \preceq (a_j,b_j)$. Since $(0,0) \preceq (a_i,b_i)$ for all $i=1,\ldots,n$, the edge set $E$ contains all the directed edges $(a_0,b_0) \to (a_i,b_i)$ with $i \not= 0$. Let $\mathscr{G}$ be the collection of all graphs $G$ with vertex set $G_V=\mathcal{V}$ and edge set $G_E \subseteq E$ such that for any $1\le j \le n$, there exists a directed path from $(a_0,b_0)$ to $(a_j,b_j)$ with each edge in $G_E$. We define a MDST on $\cal V$ as any graph $T \in \mathscr{G}$ that minimizes $\sum_{e \in G_E} |e|$ where $|e|$ denotes the Euclidean length of the edge $e$. Clearly $T$ is a tree and need not be unique. 

Now let $\cal P$ be a random collection of $n$ points uniformly and independently placed in the unit square $[0,1]^2$ in $\mathbb{R}^2$. In this random setting, the MDST on the point set $\mathcal{V}=P \cup \{(0,0)\}$ is uniquely defined almost surely, see \cite{BhRo}. By relabeling the points according to the size of their $x$-coordinate, without loss of generality, we may let the points in $\mathcal{P}$ be $(X_1,Y_1), \dots, (X_n,Y_n)$ where $Y_1, \dots ,Y_n$ are independent ${\cal U}[0,1]$ random variables, and also independent of $X_1,\ldots,X_n$, where $0<X_1< X_2 < \dots <X_n<1$ have the distribution of the order statistics generated from a sample of $n$ independent ${\cal U}[0,1]$ variables.

Though the origin is the unique minimal point of $\cal V$, the usual set of interest is the collection of minimal points of $\cal P$, which has size at least one. For $i=1,\ldots,n$, observe that $(X_i,Y_i)$ is a minimal point of $\mathcal{P}$ if and only if $Y_j>Y_i$ for all $j<i$. One much studied quantity in this context is the sum $S_n$ of the $\alpha^{th}$ powers of the Euclidean distances between the minimal points of the process and the origin for some $\alpha>0$; the work \cite{PW04} shows that $S_n$ converges to $D_{2/\alpha}$ in distribution as $n$ tends to infinity.

The lower record times $R_1,R_2,\ldots$ of the height process $Y_1,\dots,Y_n$ are also studied, see \cite{BhRo}, and are defined  by letting $R_1=1$, and for $i>1$ by
\beas
R_i = \left\{
\begin{array}{cc}
	\infty & \qmq{if $Y_j \ge Y_{R_{i-1}}$ for all $j>R_{i-1}$ or if $R_{i-1}\ge n$,} \\
	\min \{j>R_{i-1}:Y_j<Y_{R_{i-1}}\} & \text{otherwise}.
\end{array}
\right.
\enas 
In terms of these record times, the collection of the $k(n)$ minimal points inside the unit square is given by $(X_{R_i},Y_{R_i})$ for $i=1,\ldots,k(n)$.  We claim that the scaled sum of lower record times
\bea \label{eq:Wn.lower.record}
W_n = \frac{1}{n}\sum_{i=1}^{k(n)}R_i
\ena
can be approximated by the Dickman distribution ${\cal D}$ in the Wasserstein-2 metric in \eqref{def:d111} to within the bound specified by inequality \eqref{eq:ref.for.rec} of Theorem \ref{thm:weighted.Bernoullis}. Indeed, for $1 \le j \le n$, letting
\beas
B_k = \mathds{1}(k \in \{R_1,\ldots,R_{k(n)}\})
\enas
we have that $\sum_{i=1}^{k(n)} R_i=\sum_{k=1}^n k B_k$. As 
Lemma 2.1 of \cite{BhRo} shows that $B_1,\ldots,B_n$ are independent with $B_k \sim {\rm Ber}(1/k)$ for $1 \le k \le n$,  Theorem \ref{thm:weighted.Bernoullis} yields the claimed bound for the Dickman approximation of \eqref{eq:Wn.lower.record}. 

We now present the proof of our first main result.

\noindent {\em Proof of Theorem \ref{thm:weighted.Bernoullis}:}  Let $W_n$ be as in \eqref{eq:weigh.sum.Y} and take $\theta=1$ in \eqref{eq:SteinEqWn}. Letting
\beas
W_n^{(k)} = W_n-\frac{Y_k}{n}B_k,
\enas
evaluating the first term on the right hand side of \eqref{eq:SteinEqWn} yields
\begin{multline*}
E[W_nf'(W_n)]=E\left[\frac{1}{n}\sum_{k=1}^n Y_kB_k f'(W_n)\right]=\frac{1}{n}\sum_{k=1}^n E\left[Y_kB_k f'\left(W_n^{(k)}+\frac{Y_k}{n}B_k\right)\right]\\
=\frac{1}{n}\sum_{k=1}^n E\left[Y_kf'\left(W_n^{(k)}+\frac{Y_k}{n}\right)\right]P(B_k=1)=\frac{1}{n}\sum_{k=1}^n E\left[\frac{Y_k}{k}f'\left(W_n^{(k)}+\frac{Y_k}{n}\right)\right].
\end{multline*}
The right hand side of \eqref{eq:SteinEqWn} is therefore the expectation of 
\begin{multline}\label{brkdown}
	\frac{1}{n}\sum_{k=1}^n \frac{Y_k}{k}f'\left(W_n^{(k)}
	+\frac{Y_k}{n}\right) - \int_0^1 f'(W_n+u)du 
	= \frac{1}{n}\sum_{k=1}^n \frac{Y_k}{k} \left(f'\left(W_n^{(k)} +\frac{Y_k}{n}\right)-f'\left(W_n^{(k)}+\frac{k}{n}\right) \right) 
	\\+ \frac{1}{n}\sum_{k=1}^n \left(\frac{Y_k}{k} f'\left(W_n^{(k)}+\frac{k}{n}\right)-f'\left(W_n^{(k)}+\frac{k}{n}\right) \right) \\+ \frac{1}{n}\sum_{k=1}^n \left(f'\left(W_n^{(k)}+\frac{k}{n}\right) -f'\left(W_n+\frac{k}{n}\right) \right) 
	+\left(\frac{1}{n}\sum_{k=1}^n f'\left(W_n+\frac{k}{n}\right)- \int_0^1 f'(W_n+u)du\right).
\end{multline}
Using that $f \in \mathcal{H}_{1,1/2}$, and hence in particular that $f'(\cdot)$ is Lipschitz, applying the Cauchy-Schwarz inequality to the first difference on the right hand side of \eqref{brkdown} we find that the expectation of that term is bounded by
\beas
\frac{\|f''\|_\infty}{n^2}\sum_{k=1}^nE\left[\frac{|Y_k|}{k}|Y_k-k| \right] \le \frac{1}{2n^2} \sum_{k=1}^n \frac{1}{k} \sqrt{(\sigma_k^2 + k^2) \sigma_k^2}.
\enas
The expectation of the second difference is zero as $E[Y_k]=k$ and $Y_k$ is independent of $W_n^{(k)}$. 
For the expectation of the third difference, noting that $E[Y_k B_k]=1$, we similarly obtain the bound
\beas
\frac{\|f''\|_\infty}{n}\sum_{k=1}^n  E|W_n^{(k)}-W_n|\le \frac{1}{2n}\sum_{k=1}^n E\left[\frac{Y_k}{n}B_k\right]=\frac{1}{2n}.
\enas
Finally, for the fourth difference,  applying that same bound on the second derivative of $f(\cdot)$, almost surely

\begin{multline*}
\Bvert \frac{1}{n}\sum_{k=1}^n f'\left(W_n+\frac{k}{n}\right)- \int_0^1 f'(W_n+u)du \Bvert \le \sum_{k=1}^n \int_{\frac{k-1}{n}}^{\frac{k}{n}}\Bvert [f'(W_n+k/n)- f'(W_n+u)]\Bvert du \\
\le \frac{1}{2}\sum_{k=1}^n \int_{\frac{k-1}{n}}^{\frac{k}{n}} (k/n-u)du = \frac{1}{2} \left( \frac{1}{n^2}\sum_{k=1}^n k - \int_0^1 udu \right)= \frac{1}{4n}.
\end{multline*}
Combining these three bounds yields, via \eqref{eq:SteinEqWn} with $\theta=1$, that
\beas
|E[h(W_n)]-E[h(D)]|
\le \frac{3}{4n}+ \frac{1}{2n^2} \sum_{k=1}^n \frac{1}{k} \sqrt{(\sigma_k^2 + k^2) \sigma_k^2}.
\enas
Taking the supremum over ${\cal H}_{1,1}$ and recalling the definition of the norm $d_{1,1}$ in \eqref{def:d111} now yields the theorem. The final claim \eqref{eq:ref.for.rec} holds as $\sigma_k^2=0$ when $Y_k=k$ a.s. \qed

We turn now to the proof of our next main result, proceeding along the same lines as in the proof of Theorem \ref{thm:weighted.Bernoullis}. We first recall the well known Stein identity for the Poisson distribution, see e.g. \cite{Chen}, that
\bea \label{eq:Stein.for.Poisson}
P \sim {\cal P}(\lambda) \qmq{if and only if} E[P g(P)]=\lambda E[g(P+1)] 
\ena
for all functions $g(\cdot)$ on the non-negative integers for which the expectation of either side exists.

\noindent {\em Proof of Theorem \ref{thm:Poisson}:}
Consider equation \eqref{eq:SteinEqWn} with $W_n$ as in \eqref{eq:Wn.YkPk} and $h(\cdot)$ an arbitrary function in $\mathcal{H}_{1,1}$, and $f \in \mathcal{H}_{\theta,\theta/2}$ the solution of \eqref{eq:stein.eq.gen.dickman} guaranteed by Theorem \ref{fsolutionbds}. For $k=1,\ldots,n$ set $W_n^{(k)}=W_n-Y_kP_k/n$.  Using that $P_1,\ldots,P_n,Y_1,\ldots,Y_n$ are independent with $P_k \sim {\cal P}(\theta/k)$ and \eqref{eq:Stein.for.Poisson} for the second equality, letting $S_k=\{Y_j, j \in \{1,\ldots,n\}, P_j, j \in \{1,\ldots,n\} \setminus \{k\}\}$,  we have
\begin{multline*}
E[(W_n/\theta)f'(W_n)] = \frac{1}{\theta n} \sum_{k=1}^n E\left[Y_k E\left[P_k f'(W_n^{(k)}+Y_kP_k/n)|S_k\right]\right] \\
=  \frac{1}{n} \sum_{k=1}^n E\left[\frac{Y_k}{k}  E\left[f'(W_n^{(k)}+Y_kP_k/n+Y_k/n)|S_k \right]\right] =\frac{1}{n} \sum_{k=1}^n  E\left[\frac{Y_k}{k}f'\left(W_n+Y_k/n\right)\right].
\end{multline*}
Thus, via \eqref{eq:SteinEqWn}, we obtain
\begin{align}\label{Poibrkdown}
&E[h(W_n)]-E[h(D_\theta)]=E\left[(W_n/\theta)f'(W_n)-(f(W_n+1)-f(W_n)) \right]\nonumber\\
&=E\left[\frac{1}{n}\sum_{k=1}^n \frac{Y_k}{k}f'\left(W_n+\frac{Y_k}{n}\right) - \int_0^1 f'(W_n+u)du\right]\nonumber\\
&=E\left[\frac{1}{n}\sum_{k=1}^n \left(\frac{Y_k}{k}f'\left(W_n+\frac{Y_k}{n}\right) - f'\left(W_n+\frac{k}{n}\right)\right)\right]\nonumber\\
&\qquad\qquad\qquad\qquad\qquad\qquad\qquad+ E\left[ \frac{1}{n} \sum_{k=1}^n  f'\left(W_n+\frac{k}{n} \right)- \int_0^1 f'(W_n+u)du\right].
\end{align}
Now for the second term in \eqref{Poibrkdown}, since $f \in \mathcal{H}_{\theta,\theta/2}$, as for this same term that appears in the proof of Theorem \ref{thm:weighted.Bernoullis}, we have almost surely that
\bea\label{second}
\Bvert \frac{1}{n} \sum_{k=1}^n  f'(W_n+k/n) - \int_0^1 f'(W_n+u)du  \Bvert 
\le \frac{\theta}{4n}.
\ena
Now we write the first term in \eqref{Poibrkdown} as the expectation of
\begin{multline}\label{Furtherbrkdwn}
\frac{1}{n}\sum_{k=1}^n\left[\frac{Y_k}{k}f'\left(W_n+\frac{Y_k}{n}\right)- \frac{Y_k}{k}f'\left(W_n+\frac{k}{n}\right)\right]\\
+\frac{1}{n}\sum_{k=1}^n \left[\frac{Y_k}{k}f'\left(W_n+\frac{k}{n}\right)- f'\left(W_n+\frac{k}{n}\right)\right] .
\end{multline}
As in proof of Theorem \ref{thm:weighted.Bernoullis}, recalling that $f \in \mathcal{H}_{\theta,\theta/2}$, the expectation of the first term in \eqref{Furtherbrkdwn} is bounded by
\beas
\frac{\|f''\|_\infty}{n^2}\sum_{k=1}^nE\left[\frac{|Y_k|}{k}|Y_k-k| \right] \le \frac{\theta}{2n^2} \sum_{k=1}^n \frac{1}{k} \sqrt{(\sigma_k^2 + k^2) \sigma_k^2}.
\enas
The expectation of the second term in \eqref{Furtherbrkdwn} can be bounded by
\beas
\frac{\|f'\|_\infty}{n}\sum_{k=1}^n E\Bvert \frac{Y_k}{k}-1 \Bvert \le \frac{\theta}{n} \sum_{k=1}^n \frac{\sigma_k}{k}.
\enas
Assembling the bounds on the terms arising from \eqref{Poibrkdown}, consisting of \eqref{second} and the two inequalities above, we obtain
\beas
|E[h(W_n)]-E[h(D_\theta)]| \le \frac{\theta}{4n} +\frac{\theta}{n} \sum_{k=1}^n \frac{\sigma_k}{k}+ \frac{\theta}{2n^2} \sum_{k=1}^n \frac{1}{k} \sqrt{(\sigma_k^2 + k^2) \sigma_k^2}.
\enas
Taking the supremum over $h \in {\mathcal H}_{1,1}$ and applying definition \eqref{def:d111} completes the proof of \eqref{PoisY}. The inequality in \eqref{Pois} follows by observing that $\sigma_k^2=0$ when $Y_k=k$ a.s.
\qed

\subsection{Dickman approximation in number theory}\label{sub:pnt}
Let $(p_k)_{k \ge 1}$ be an enumeration of the prime numbers in increasing order. Let $(X_k)_{k \ge 1}$ be a sequence of independent integer valued random variables and let
\bea\label{def:sn}
S_n=\frac{1}{\log(p_n)} \sum_{k=1}^n X_k \log(p_k) \quad \mbox{for $n \ge 1$.}
\ena 
Weak convergence of $S_n$ to the Dickman distribution in the cases when the $X_k$'s are distributed as geometric and Bernoulli variables is well known in probabilistic number theory, and \cite{AMPS} recently provided a rate of convergence in the Bernoulli case. We give bounds in a stronger metric and remove a logarithmic factor from their rate. We also prove such bounds when the $X_k$'s are distributed as geometric or Poisson with parameters given by certain functions of $p_k$. For our results in this area, we rely heavily on the techniques in the proof of Lemma \ref{lem:swan} of  \cite{AMPS}; in particular, the
identity \eqref{steincoeff} below, without remainder, is due to \cite{AMPS}. We begin with the following abstract theorem.

\begin{theorem}\label{thm:numgen}
Let $S$ be a non-negative random variable with finite variance such that for some constant $\mu$ and a random variable $T$ satisfying $P(S+T=0)=0$,
\bea\label{steincoeff}
E[S \phi (S)]=\mu E[\phi(S+T)]+ R_\phi \qmq{for all $\phi \in {\rm Lip}_{1/2}$,}
\ena
where the constant $R_\phi(\cdot)$ may depend on $\phi(\cdot)$. Then 
\bea \label{eq:inf.TU}
d_{1,1}(S,D) \le |\mu-1| + \frac{1}{2} \inf_{(T,U)} E|T - U| +\sup_{\phi \in {\rm Lip_{1/2}}} |R_\phi|
\ena
where $D$ is a standard Dickman random variable, and the infimum is over all couplings $(T,U)$ of $T$ and $U\sim {\cal U}[0,1]$ constructed on the same space as $S$, with $U$ independent of $S$.
\end{theorem}

\begin{remark}
We note the connection between the relation in \eqref{steincoeff} and size biasing, where for a non-negative random variable $S$ with finite mean $\mu$, we say $S^s$ has the $S$-size biased distribution when 
		\beas 
		E[S \phi (S)]=\mu E[\phi(S^s)]
		\enas
		for all functions $\phi(\cdot)$ for which these expectations exist. In particular, when $R_\phi$ in \eqref{steincoeff} is zero for all $\phi \in {\rm Lip}_{1/2}$, we obtain that $S^s=_d S+T$; for an application which requires the remainder, see Lemma \ref{lem:geom}. Additionally, Section 4.3 of \cite{ABT} shows that the standard Dickman $D$ is the unique non-negative solution to the distributional equality $W^s =_d W+U$, where $U$ is ${\cal U}[0,1]$, and independent of $W$. Hence, the error term comparing $T$ and $U$ in Theorem \ref{thm:numgen} is natural.
\end{remark}

\noindent {\em Proof of Theorem \ref{thm:numgen}:}
	We first show that the set of couplings over which the infimum is taken in \eqref{eq:inf.TU} is non-empty. Note that the case when $S$ is identically zero is trivial since one can take $\mu=0$, $T=0$ and $R_\phi=0$ for all $\phi \in {\rm Lip}_{1/2}$. For a nontrivial $S$, let $\mu=E[S]$, and let $S^s$ and $U$ be constructed on the same space as $S$, independently of $S$, with $S^s$ having the $S$-size biased distribution and $U \sim {\cal U}[0,1]$. Then setting
	$T=S^s-S$ identity \eqref{steincoeff} is satisfied with $R_\phi=0$ for all $\phi \in {\rm Lip}_{1/2}$, and the pair $(T,U)$ satisfies the conditions required of the infimum in the theorem.
	
	 Invoking Theorem \ref{fsolutionbds} with $\theta=1$, for any given $h \in {\cal H}_{1,1}$ there exists a function $f(\cdot)$ satisfying $\|f'\|_{(0,\infty)} \le 1$ and $\|f''\|_{(0,\infty)} \le 1/2$ such that
	\beas
	E[h(S)]-E[h(D)]=E[Sf'(S)+f(S)-f(S+1)].
	\enas
	Now consider $\mu$ and
	$T$ satisfying \eqref{steincoeff} with $(T,U)$ constructed on the same space as $S$, with $U \sim {\cal U}[0,1]$ and independent of $S$. Then, using $P(S+T=0)=0$, allowing us to apply the bounds of Theorem \ref{fsolutionbds} over $(0,\infty)$, the mean value theorem for the second inequality and recalling definitions \eqref{esssup:def} and \eqref{supnorm}, we obtain
	\begin{multline*}
	|E[h(S)]-E[h(D)]|=|E[Sf'(S) - f'(S+U)]|=|E[\mu f'(S+T)-f'(S+U)+R_{f'}]|\\
	\le |E[\mu f'(S+T)-f'(S+T)]| + |E[f'(S+T)-f'(S+U)]| +|R_{f'}| \\
	\le \|f'\|_{(0,\infty)} |\mu-1| + \|f''\|_{(0,\infty)} E|T - U| + |R_{f'}| \le |\mu-1| +  \frac{1}{2}E|T - U| + |R_{f'}|.
	\end{multline*}
	Now taking the infimum on the right hand side over all couplings $(T,U)$ satisfying the conditions of the theorem  yields
	\beas
	|E[h(S)]-E[h(D)]| \le |\mu-1| + \frac{1}{2} \inf_{(T,U)} E|T - U| + |R_{f_{h}'}|,
	\enas
	where we have written $f=f_h$ to emphasize the dependence of $f(\cdot)$ on $h(\cdot)$.
	Taking supremum over $h \in {\cal H}_{1,1}$ first on the right, and then on the left now yields the result upon applying definition \eqref{def:d111}. 
	\qed
	
Now we will demonstrate a few applications of Theorem \ref{thm:numgen}. In all these examples the conditions that the variance of $S$ is finite and that $S+T>0$ almost surely are straightforward to check, and will not be mentioned further. For $n \ge 1$, let $\Omega_n$ denote the set of integers with no prime factor larger than $p_n$, and let $\Pi_n$ be the distribution on $\Omega_n$ with mass function
\beas
\Pi_n (m)=\frac{1}{\pi_nm}  \qmq{for $m \in \Omega_n$}
\enas 
where $\pi_n=\sum_{m \in \Omega_n} 1/m$ is the normalizing factor. One can check, see e.g. Proposition 1 in \cite{Pi16}, that $M_n=\prod_{k=1}^n p_k^{X_k}$ has distribution $\Pi_n$, where $X_k \sim {\rm Geom}(1-1/p_k)$ are independent for $1 \le k \le n$ ; we remind the reader that we write $X \sim {\rm Geom}(p)$ when $P(X=m)=(1-p)^m p$ for $m \ge 0$. For $n \ge 1$, the random variable $S_n$ as in \eqref{def:sn} is therefore given by
\bea \label{eq:Sn.primes}
S_n=\frac{1}{\log(p_n)} \sum_{k=1}^n X_k \log(p_k)=\frac{\log M_n}{\log (p_n)}.
\ena
Taking the mean, we find
\bea \label{eq:mun.geom}
\mu_n=E[S_n]= \frac{1}{\log(p_n)} \sum_{k=1}^n \frac{\log(p_k)}{p_k-1}.
\ena
Now define the random variable $I$ taking values in $\{1, \dots, n\}$, and independent of $S_n$, with mass function
\bea \label{eq:defI.geom}
P(I=k)=\frac{\log(p_k)}{(p_k-1) \log(p_n)\mu_n} \qmq{for $k \in \{1, \dots, n\}$.}
\ena
The next lemma very closely follows the arguments in Lemmas 3 and 5 of \cite{AMPS} and is included here only for completeness. In the proof, we will use the statement, equivalent \cite{HW79} to the prime number theorem, that $\lim_{n \to \infty} p_n/(n \log n) = 1$, and Rosser's Theorem \cite{Ro39}, to respectively yield that
	\bea\label{logpn}
	\log p_n = \log n + O(\log \log n) \qmq{and} p_k > k \log k.
\ena 
We will also use the follwing stronger version of Merten's theorem, see \cite{Fi03}: For $j \ge 1$ and $\gamma$ the Euler constant,
\bea\label{Mertens}
\sum_{k=1}^{j} \log (p_k)/p_k = \log (p_{j}) + R_j \qmq{with} \lim_{j \to \infty} R_j=-\gamma-\sum_{k=1}^{\infty} \frac{\log (p_k)}{(p_k-1)p_k}=-1.33\dots
\ena
\begin{lemma}\label{lem:geom}
	Let $S_n$ be as in \eqref{eq:Sn.primes} with $X_1,\ldots,X_n$ independent with $X_k \sim {\rm Geom}(1-1/p_k)$, $\mu_n$ as in \eqref{eq:mun.geom}, $I$ with distribution given in \eqref{eq:defI.geom} and independent of $S_n$ and 
	\beas
	T_n=\frac{\log(p_I)}{\log(p_n)} \qmq{and} R_{n,\phi}=\frac{1}{\log(p_n)} \sum_{k=1}^n \frac{ \log(p_k)}{p_k-1}E\left[X_k \left(\phi\left(S_n+ \frac{\log(p_k)}{\log(p_n)}\right)-\phi(S_n)\right)\right].
	\enas
	Then
	\beas
	E[S_n \phi (S_n)]=\mu_n E[\phi(S_n+T_n)] + R_{n,\phi} \qmq{for all $\phi \in {\rm Lip}_{1/2}$.}
	\enas
	Moreover
	\beas
	\sup_{\phi \in {\rm Lip}_{1/2}} |R_{n,\phi}|=O\left(\frac{1}{\log^2 n}\right) \qmq{and} \mu_n-1=O\left(\frac{1}{\log n}\right),
	\enas
	and there exists a coupling between $U \sim {\cal U}[0,1]$ and $T_n$ with $U$ independent of $S_n$, such that
	\beas
	E|T_n-U|=O\left(\frac{1}{\log n}\right).
	\enas
\end{lemma}
	\begin{proof}
	It is easily verified that for $X \sim {\rm Geom}(p)$,
	\bea \label{geomstid}
	E[g(X)]=\frac{1-p}{p}E[g(X+1)-g(X)]
	\ena
	for all functions $g(\cdot)$ for which these expectations exist, and which satisfy $g(0)=0$.
	Let $S_n^{(k)}=S_n - X_k \log(p_k)/\log(p_n)$. Since $X_k \sim {\rm Geom} (1-1/p_k)$, specializing \eqref{geomstid} to the case $g(x)=x\phi(S_n^{(k)}+x \log(p_k)/ \log(p_n))$, conditioning on $S_n^{(k)}$ in the second equality and using the independence of $I$ and $S_n$ in the last, for $\phi \in {\rm Lip}_{1/2}$ we have
	\begin{multline*}
	E[S_n \phi (S_n)]=\frac{1}{\log(p_n)} \sum_{k=1}^n \log(p_k) E\left[X_k \phi\left(S_n^{(k)}+X_k \frac{\log(p_k)}{\log(p_n)}\right)\right]\\
	=\frac{1}{\log(p_n)} \sum_{k=1}^n \frac{\log(p_k)(1/p_k)}{1-1/p_k} E\left[(X_k+1) \phi\left(S_n+ \frac{\log(p_k)}{\log(p_n)}\right)-X_k\phi(S_n) \right]\\
	=\frac{1}{\log(p_n)} \sum_{k=1}^n \frac{\log(p_k)}{p_k-1} E\left[\phi\left(S_n+ \frac{\log(p_k)}{\log(p_n)}\right) + X_k \left(\phi\left(S_n+ \frac{\log(p_k)}{\log(p_n)}\right)-\phi(S_n)\right) \right]\\
	=\mu_n \sum_{k=1}^n \frac{\log(p_k)}{(p_k-1)\log (p_n)\mu_n} E\left[\phi\left(S_n+ \frac{\log(p_k)}{\log(p_n)}\right)\right] + R_{n,\phi}\\
	=\mu_n \sum_{k=1}^n P(I=k) E\left[\phi\left(S_n+ \frac{\log(p_k)}{\log(p_n)}\right)\right] + R_{n,\phi}=\mu_n E[\phi(S_n+T_n)] + R_{n,\phi},
	\end{multline*}
	proving the first claim. 
	
	Next, using mean value theorem and that $\|\phi'\|_\infty \le 1/2$ in the first inequality, we have
	\begin{multline}
	|R_{n,\phi}|=\Bvert\frac{1}{\log(p_n)} \sum_{k=1}^n \frac{\log(p_k)}{p_k-1} E\left[X_k \left(\phi \left(S_n+ \frac{\log(p_k)}{\log(p_n)}\right)-\phi(S_n)\right)\right]\Bvert\\
	\le \frac{1}{2\log(p_n)} \sum_{k=1}^n \frac{ \log^2(p_k)}{(p_k-1)\log(p_n)} EX_k=\frac{1}{2\log^2 (p_n)} \sum_{k=1}^n \frac{\log^2(p_k)}{(p_k-1)^2}=O\left(\frac{1}{\log^2 n}\right)
	\label{eq:see.final.sum}
	\end{multline}
	where in the last step, we have used that the second relation in \eqref{logpn} to lower bound $p_n$ by $n$, and, again by \eqref{logpn}, that
		\beas
	\sum_{k=1}^\infty \frac{\log^2(p_k)}{(p_k-1)^2} \le C \sum_{k=1}^\infty\frac{ \log^2(k)}{(k-1)^2}<\infty,
	\enas
	where we have used the first relation there to upper bound $\log (p_k)$ by $ C \log (k)$ for some positive constant in the numerator, and the second one again to lower bound $p_k$ by $k$ in the denominator. As the final sum in \eqref{eq:see.final.sum} does not depend on $\phi(\cdot)$, the bound is uniform over all $\phi \in {\rm Lip}_{1/2}$.

	The proof of the remainder of the lemma closely follows Lemma 5 of \cite{AMPS}. Using \eqref{Mertens}, we obtain
		\bea\label{sum'}
		\sum_{k=1}^{j} \frac{\log (p_k)}{p_k-1}=\log (p_j) + \sum_{k=1}^{j} \frac{\log (p_k)}{(p_k-1)p_k}+R_j=\log (p_{j}) + O(1),
		\ena
		where in the second sum we have used both relations in \eqref{logpn} to obtain $\frac{\log (p_k)}{p_k(p_k-1)}=O\left(\frac{1}{k^2 \log k}\right)$. Thus, using \eqref{sum'}, that $p_n>n$ via \eqref{logpn}, and recalling $\mu_n$ in \eqref{eq:mun.geom}, we obtain
		\bea\label{mu'}
		\mu_n-1=\frac{1}{\log (p_n)}\left(\sum_{k=1}^n \frac{\log (p_k)}{p_k-1}-\log (p_n)\right)=O\left(\frac{1}{\log n}\right).
		\ena
		
		To prove the last claim, we sketch the coupling construction of $(U, I)$ in Lemma 5 of \cite{AMPS}, with $I$ a function of the uniform $U\sim {\cal U}[0,1]$, itself independent of $X_1,\ldots,X_n$. For $j=0,1,\ldots,n$, set 
		\beas
		F_j=\sum_{k=1}^j P(I=k)=\frac{1}{\mu_n \log (p_n)} \sum_{k=1}^j \frac{\log(p_k)}{(p_k-1)}, 
		\enas
		and define the random variable $I$ by
		\beas
		I=j \qmq{if} F_{j-1} \le U < F_j.
		\enas 
		Clearly $I$ is independent of $X_1,\ldots,X_n$, since it only depends on $U$.
		When $I=j$, using $|u-c|$ is a convex function of $u$ for any constant $c$ for the equality, deterministically we have
		\bea\label{detbd} 
		\Bvert U- \frac{\log(p_I)}{\log(p_n)}\Bvert \le \sup_{u \in [F_{j-1}, F_j)} \Bvert u-\frac{\log(p_j)}{\log (p_n)} \Bvert = \max \Big\{\Bvert F_{j-1}-\frac{\log(p_j)}{\log (p_n)} \Bvert,\Bvert F_{j}-\frac{\log(p_j)}{\log (p_n)} \Bvert \Big\}.
		\ena
		Now, using \eqref{logpn}, \eqref{sum'} and \eqref{mu'}, with \eqref{mu'} implying that $\mu_n \to 1$ as $n \to \infty$, 
		we have
		\begin{multline*}
		F_{j}-\frac{\log(p_j)}{\log (p_n)} = \frac{1}{\log (p_n)}\left(\sum_{k=1}^j \frac{\log (p_k)}{p_k-1}-\log (p_j)\right) - \frac{1}{\log p_n}(1-\mu_n^{-1}) \sum_{k=1}^j \frac{\log (p_k)}{p_k-1}\\
		= O\left(\frac{1}{\log n}\right) - \frac{\mu_n -1}{\mu_n \log p_n} (\log p_j + O(1))= O\left(\frac{1}{\log n}\right).
		\end{multline*}
		Also, using \eqref{logpn} and again that $\mu_n \rightarrow 1$, we have
		\bea \label{eq:PIeqj}
		P(I=j)= F_j-F_{j-1} =\frac{1}{\mu_n \log (p_n)} \frac{\log(p_j)}{(p_j-1)} = O\left(\frac{1}{j \log n}\right).
		\ena
		Thus, by subtracting and adding $F_j$, we obtain
		\beas
		F_{j-1}-\frac{\log(p_j)}{\log (p_n)}= O\left(\frac{1}{j \log n}\right) + O\left(\frac{1}{\log n}\right) = O\left(\frac{1}{ \log n}\right),
		\enas
		and hence, on the event $I=j$, from \eqref{detbd} we have
		\bea \label{eq:sup.interval.bd}
	\Bvert U- \frac{\log(p_I)}{\log(p_n)}\Bvert = O\left(\frac{1}{ \log n}\right).
		\ena
		Now, using \eqref{eq:PIeqj} and \eqref{eq:sup.interval.bd} we obtain
		\begin{multline*}
		E\Bvert U - \frac{\log(p_I)}{\log(p_n)} \Bvert = \sum_{j=1}^n P(I=j) E\left[\Bvert U - \frac{\log(p_I)}{\log(p_n)} \Bvert \Bvert I=j\right]\\
		= O\left(\sum_{j=1}^n \frac{1}{j \log n} \frac{1}{ \log n} \right) = O\left(\frac{1}{\log n}\right),
		\end{multline*}
		thus proving the final claim.
\end{proof}

\noindent {\em Proof of Theorem \ref{thm:num1}:} The upper bound follows directly from Theorem \ref{thm:numgen} upon invoking Lemma \ref{lem:geom}. Next we show that the order of the bound is optimal. From \eqref{mu'} and \eqref{sum'}, we have
	\beas
	\mu_n-1=\frac{1}{\log (p_n)}\left(\sum_{k=1}^{n} \frac{\log (p_k)}{(p_k-1)p_k}+R_n\right),
	\enas
 and by the second display in \eqref{Mertens} we obtain
 \beas
 \lim_{n \to \infty} \sum_{k=1}^{n} \frac{\log (p_k)}{(p_k-1)p_k}+R_n=-\gamma.
 \enas
As $\log p_n=O(\log n)$ by \eqref{logpn}, $|\mu_n-1|$ is at least of order $1/\log n$. Since the function $h(x)=x$ is in ${\cal H}_{1,1}$, by \eqref{def:d111} we have that $d_{1,1}(S_n,D) \ge |E h(S_n) - E h(D)|=|\mu_n-1|$. Hence $d_{1,1}(S_n,D)$ is at least of order $1/\log n$.\qed

For our next example, for $n \ge 1$ let $\Omega'_n$ denote the set of square-free integers whose largest prime factor is less than or equal to $p_n$ and let $\Pi'_n$ denote the distribution on $\Omega'_n$ with mass function
\beas
\Pi'_n (m)=\frac{1}{\pi'_nm}  \qmq{for $m \in \Omega'_n$}
\enas 
where $\pi'_n=\sum_{m \in \Omega'_n} 1/m$ is the normalizing factor. We again consider $S_n$ as in \eqref{eq:Sn.primes}, here for $M_n=\prod_{k=1}^n p_k^{X_k}$ where $X_k \sim {\rm Ber}(1/(1+p_k))$ are independent for $1 \le k \le n$. One can check, see e.g. \cite{CS13}, that $M_n \sim \Pi'_n$.
Following \cite{AMPS}, let
\bea\label{eq:mun.bern}
\mu_n=E[S_n]= \frac{1}{\log(p_n)} \sum_{k=1}^n \frac{\log(p_k)}{1+p_k}.
\ena
The following lemma combines Lemmas 3 and 5 of \cite{AMPS}. By following tightly the same lines of argument in \cite{AMPS} the bounds we obtain in \eqref{mun} and \eqref{Tn} are $O(1/\log n)$ whereas \cite{AMPS} claims only the order $O(\log \log n/\log n)$.
\begin{lemma}\label{lem:swan}
	Let $S_n$ be as in \eqref{eq:Sn.primes} with $X_1,\ldots,X_n$ independent with $X_k \sim {\rm Ber}(1/(1+p_k))$. With $\mu_n$ as given in \eqref{eq:mun.bern}, let the random variable $I$ take values in $\{1, \dots, n\}$ with mass function
	\beas
	P(I=k)=\frac{\log(p_k)}{(1+p_k) \log(p_n)\mu_n} \qmq{for $k \in \{1, \dots, n\}$,}
	\enas 
	and be independent of $X_1, \dots, X_n$. For
	\bea \label{eq:Tn.Bern}
	T_n=\frac{\log(p_I)}{\log(p_n)}-\frac{X_I \log(p_I)}{\log(p_n)},
	\ena
	we have
	\bea\label{SBBern} 
	E[S_n \phi (S_n)]=\mu_n E[\phi(S_n+T_n)] \qmq{
		for all $\phi \in  {\rm Lip}_{1/2}$.}
 	\ena
	Moreover, 
	\bea\label{mun}
	\mu_n-1=O\left(\frac{1}{\log n}\right) \qmq{and} E\Bvert \frac{X_I \log(p_I)}{\log(p_n)} \Bvert = O\left(\frac{1}{\log ^2 n}\right),
	\ena
	and there exists a coupling between a random variable $U \sim {\cal U}[0,1]$ and $I$ with $U$ independent of $S_n$ such that
	\bea\label{Tn}
	E\Bvert U - \frac{\log(p_I)}{\log(p_n)} \Bvert =O\left(\frac{1}{\log n}\right).
	\ena
\end{lemma}
\begin{proof}
	The proof of \eqref{SBBern} is exactly same as in Lemma 3 of \cite{AMPS} and one can follow the lines of argument in \cite{AMPS} to prove the second claim in \eqref{mun}. The proofs of the other two claims are similar to those of the corresponding results in Lemma \ref{lem:geom} noting that the orders in the bounds do not change if we replace $p_k-1$ by $p_k+1$; we omit the computation.
	\end{proof}
\noindent {\em Proof of Theorem \ref{thm:num}:} The upper bound follows directly from Theorem \ref{thm:numgen} upon invoking Lemma \ref{lem:swan} with $R_\phi=0$ for all $\phi \in {\rm Lip}_{1/2}$ and noting that with $T_n$ and $U$ as in \eqref{eq:Tn.Bern} and \eqref{Tn} respectively,
	\beas
	E|T_n - U| \le E\Bvert \frac{X_I \log(p_I)}{\log(p_n)} \Bvert + E \Bvert U- \frac{\log(p_I)}{\log(p_n)}\Bvert = O\left(\frac{1}{\log n}\right),
	\enas
	using \eqref{mun} and \eqref{Tn} on these two terms, respectively. Finally, that the upper bound is of optimal order follows as in the proof of Theorem \ref{thm:num1}.
\qed

We also prove that these types of convergence results hold for $S_n$ given in \eqref{eq:Sn.primes} when $X_k \sim {\rm Poi}(\lambda_k)$, $k \ge 1$ for certain sequences of positive real numbers $(\lambda_k)_{k \ge 1}$. Here we take $\mu_n$ equal to the mean of $S_n$,
\bea\label{poi} 
\mu_n= \frac{1}{\log(p_n)} \sum_{k=1}^n \lambda_k \log(p_k) \qmq{and} 
P(I=k)=\frac{\lambda_k \log(p_k)}{ \log(p_n)\mu_n} \qmq{for $k \in \{1, \dots, n\}$,} 
\ena
with $I$ independent of $S_n$. Under this framework, we have the following construction of a variable having the size bias distribution of $S_n$.

\begin{lemma}\label{lem:2.3}
	For a sequence of positive real numbers $(\lambda_k)_{1 \le k \le n}$ and independent random variables $X_1,\ldots,X_n$ with $X_k \sim {\rm Poi}(\lambda_k)$, let
	\beas
	S_n=\frac{1}{\log(p_n)} \sum_{k=1}^n X_k \log(p_k).
	\enas
	For $\mu_{n}$ as in \eqref{poi} and $T_n=\log(p_I)/\log(p_n)$, where $I$ is distributed as in \eqref{poi} and is independent of $S_n$,
	we have
	\beas
	E[S_n \phi (S_n)]=\mu_n E[\phi(S_n+T_n)] \qm{for all $\phi \in {\rm Lip}_{1/2}$.}
	\enas
\end{lemma}
\begin{proof}
	Using \eqref{eq:Stein.for.Poisson} in the second equality, for $S_n^{(k)}=S_n - X_k \log(p_k)/\log(p_n)$,
	\begin{multline*}
	E[S_n \phi (S_n)]=\frac{1}{\log(p_n)} \sum_{k=1}^n \log(p_k) E[X_k \phi(S_n^{(k)}+X_k \log(p_k)/\log(p_n)]\\=\frac{1}{\log(p_n)} \sum_{k=1}^n \log(p_k) \lambda_k E[ \phi(S_n^{(k)}+(X_k+1) \log(p_k)/\log(p_n)]\\=\frac{1}{\log(p_n)} \sum_{k=1}^n \log(p_k) \lambda_k E[ \phi(S_n+ \log(p_k)/\log(p_n)]\\
	=\mu_n \sum_{k=1}^n P(I=k) E[ \phi(S_n+ \log(p_k)/\log(p_n)]=\mu_n E[\phi(S_n+T_n)]
	\end{multline*}
	where in the last step, we have used that $I$ is independent of $S_n$.
\end{proof}

We now present two applications of Lemma \ref{lem:2.3} with notation and assumptions as there.
\begin{example}\label{ex:Poisson1}
	Let $\lambda_k=1/(1+p_k)$. As the mean of the $X_k$ variables are the same here as in Lemma \ref{lem:swan}, $\mu_n$ and the distribution of $I$ also correspond. Taking $U \sim {\cal U}[0,1]$ independent of $S_n$, and coupling $I$ and $U$ similarly as in Lemma \ref{lem:swan}, we have that
    \beas
    |\mu_n-1|=O\left(\frac{1}{\log n}\right) \qmq{and}  E\Bvert U - \frac{\log(p_I)}{\log(p_n)} \Bvert =O\left(\frac{1}{\log n}\right).
    \enas
    Now, by Theorem \ref{thm:numgen} and Lemma \ref{lem:2.3} we obtain
	\beas
	d_{1,1}(S_n, D) \le \frac{C}{\log n}
	\enas
	for some universal constant $C$. One may show that the order of this bound is optimal by arguing as in the proof of Theorem \ref{thm:num1}.
\end{example}

\begin{example}\label{ex:Poisson2}
	Let $p_0=1$ and and $\lambda_k=1 - \log(p_{k-1})/\log(p_k)$ for $k \ge 1$.  Then clearly $\mu_n=1$. Now to obtain a coupling $(T_n,U)$, we take $U \sim {\cal U}[0,1]$ independent of $S_n$, and define
    \beas
	I = k \qm{if \quad $\frac{\log(p_{k-1})}{\log(p_n)} \le U <\frac{\log(p_{k})}{\log(p_n)}$ for $1\le k \le n$}.
	\enas 
	Then by construction we have
	\beas
	P(I=k)=\frac{\lambda_k \log(p_k)}{ \log(p_n)\mu_n} \qmq{for $1 \le k \le n$.}
	\enas 
	Conditioning on $I$, we have
	\beas
	E|T_n-U| =\sum_{k=1}^n P(I=k) E \left(\Big| \frac{\log(p_{k})}{\log(p_n)}-U \Big| \Bvert I=k \right)
	\le \sum_{k=1}^n P(I=k) \Bvert \frac{\log(p_{k-1})}{\log(p_n)}-\frac{\log(p_k)}{\log(p_n)} \Bvert.
	\enas
	Now using that $p_k/p_{k-1} \le 2$ by Bertrand's postulate (see e.g. \cite{Ra19}) for all $k \ge 1$, we obtain
	\beas
	E|T_n-U| \le \frac{\log(2)}{\log(p_n)}.
	\enas
	Hence from Theorem \ref{thm:numgen} with $\mu_n=1$ and $R_\phi=0$ for all $\phi \in {\rm Lip}_{1/2}$, we have
	\beas
	d_{1,1}(S_n, D) \le \frac{\log (2)}{2 \log (p_n)} \le \frac{C}{\log n}
	\enas
	for some universal constant $C$.
\end{example}
Following the distribution of a draft of this manuscript, \cite{Arr17} pointed out that the approach in \cite{Arr02} may be used to obtain bounds in the Wasserstein-1 metric for some results in this section.

\section{Perpetuities and the ${\cal D}_{\theta,s}$ family, simulations and distributional bounds}\label{sec:perp}

In this section we develop the extension of the generalized Dickman distribution to the ${\cal D}_{\theta,s}$ family for $\theta>0$ and a function $s:[0,\infty) \rightarrow [0,\infty)$. As detailed in the Introduction, the recursion 
\eqref{eq:sim.dickman.theta} associated with the ${\cal D}_\theta$ family can be interpreted as giving the successive values of a Vervaat perpetuity under the assumption that the utility function is the identity. More generally, with
utility function $s(\cdot)$, one obtains the recursion
\bea \label{eq:tVervaat.insec}
s(W_{n+1})=U_n^{1/\theta}s(W_n+1) \qm{for $n \ge 0$,}
\ena
where $U_n, n \ge 0$ are independent and have the ${\cal U}[0,1]$ distribution, $U_n$ is independent of $W_n$, and $W_0$ has some given initial distribution. In Section \ref{existence.D}, under Condition \ref{cond:s} below on $s(\cdot)$, we prove Theorem \ref{existence} that shows that the distributional fixed points ${\cal D}_{\theta,s}$ of \eqref{eq:tVervaat.insec} exist and are unique. When $s(\cdot)$ is invertible, as it is under Condition \ref{cond:s} below, we may write \eqref{eq:tVervaat.insec} as
\bea \label{eq:Wn.t.rec}
W_{n+1}=s^{-1}\left( U_n^{1/\theta}s(W_n+1)\right) \qm{for $n \ge 0$.}
\ena

In Section \ref{bounds}, we provide distributional bounds for approximation of the ${\cal D}_{\theta,s}$ distribution. Using direct coupling, Corollary \ref{cor:dist.sW.sD} gives a bound on how well the utility $s(W_n)$ in \eqref{eq:tVervaat.insec} approximates the utility of its limit $D_{\theta,s}$. Next, Theorem \ref{bound} extends the main Wasserstein bound \eqref{eq:bound17} of \cite{Go17v2} to
\bea \label{eq:Wn.t.rec*}
d_1(W,D_{\theta,s}) \le (1-\rho)^{-1} d_1(W^*,W)
\qmq{where} W^*=_d s^{-1}\left( U^{1/\theta}s(W+1)\right)
\ena
for $U \sim {\cal U}[0,1]$, independent of $W$. The constant $\rho$ is defined in \eqref{rholess1} as a uniform bound on an integral involving $(\theta,s)$ given by \eqref{I:def}. However, \cite{Bax17} shows that this quantity can be interpreted in terms of the Markov chain \eqref{eq:Wn.t.rec} and its properties connected to those of its transition operator $(Ph)(x)=E\left[h\left( s^{-1}\left( U^{1/\theta}s(x+1)\right)\right)\right]$ in this, and some more general, cases. In particular, for $h \in {\rm Lip_1}$, $\rho$ is a bound on the essential supremum norm of the derivative of the transition operator. Though linear stochastic recursions are ubiquitous and are well known to be highly tractable, this special class of Markov chains, despite its non-linear transitions, seems also amenable to deeper analysis.

We apply the inequality \eqref{eq:Wn.t.rec*} in Corollary \ref{othertwo} to obtain a bound on the Wasserstein distance between the iterates $W_n$ of \eqref{eq:Wn.t.rec} and $D_{\theta,s}$. Finally in Section \ref{subsec:dickman.examples}, we give a few examples of some new distributions that arise as a result of utility functions that appear in the economics literature.

\subsection{Existence and uniqueness of ${\cal D}_{\theta,s}$ distribution}\label{existence.D}
In the following we use the terms increasing and decreasing in the non-strict sense. Let $\le_{\rm st}$ denote inequality between random variables in the stochastic order.

 \begin{lemma} \label{lem:sto.order.pres}
 	Let $\theta>0$ and $s:[0,\infty) \rightarrow [0,\infty)$ satisfy
 	\bea\label{Subadd}
 	s(x+1) \le s(x)+1 \qmq{for all $x \ge 0$,}
 	\ena
    let $W_0$ be a given non-negative random variable and let $\{W_n, n \ge 1\}$ be generated by recursion \eqref{eq:tVervaat.insec}.
 	Then 
  	\bea \label{eq:sw.sub.add}
 	s(W_{n+1})  \le U_n^{1/\theta} (s(W_n)+1)\qm{for all $n \ge 0$.}
 	\ena
 If in addition $s(W_0) \le_{st} D_\theta$, then
 	 \bea \label{eq:sWn.le.D}
 	 s(W_n) \le_{\rm st} D_\theta  
 	 \qm{for all $n \ge 0$.}
 	 \ena
 \end{lemma}
 
 \begin{proof}
By applying \eqref{eq:tVervaat.insec} and \eqref{Subadd} for the equality and inequality respectively, we have
 	\beas
s(W_{n+1}) = U_n^{1/\theta} s(W_n+1) \le U_n^{1/\theta} (s(W_n)+1) ,
\enas
hence the claim \eqref{eq:sw.sub.add} holds, and when $s(W_n) \le_{\rm st} D_\theta$ then
\beas
U_n^{1/\theta} (s(W_n)+1) \le_{\rm st} U_n^{1/\theta} (D_\theta+1) =_d D_\theta,
\enas
where for the final equality we have used that $D_\theta$ is fixed by the Dickman bias transformation \eqref{def:W*}, and taken $U_n$ independent of $D_\theta$. Induction then shows that the claim holds for all $n \ge 0$ when \eqref{eq:sWn.le.D} is true for $n=0$.
 \end{proof}
Theorem \ref{existence}, showing the existence and uniqueness of the fixed point ${\cal D}_{\theta,s}$ to \eqref{bias}, requires the following condition to hold on the utility function $s(\cdot)$.
\begin{condition} \label{cond:s}
	The function $s:[0,\infty) \rightarrow [0,\infty)$ is continuous, strictly increasing with $s(0)=0$ and $s(1)=1$, and satisfies
	\bea\label{Subadditivity}
	s(x+1) \le s(x)+1 \qmq{for all $x \ge 0$}
	\ena
	and
	\bea\label{Concavity}
	|s(x+1)-s(y+1)| \le |s(x)-s(y)| \qmq{for all $x,y \ge 0$.}
	\ena
\end{condition}

The following result shows that choice of the starting distribution in \eqref{eq:tVervaat.insec} has vanishing effect asymptotically as measured in the $d_1$ Wasserstein norm. 

\begin{lemma} \label{lem:was.s(W)s(V)}
Let $\theta>0$ and Condition \ref{cond:s} be in force. Let $W_0$ and $V_0$ be given non-negative random variables such that the means of $s(W_0)$ and $s(V_0)$ are finite. For $n \ge 1$ let  $s(V_n)$ and $s(W_n)$ have distributions as specified in \eqref{eq:tVervaat.insec}. 
 Then $s(W_n)$ and $s(V_n)$ have finite mean for all $n \ge 0$, and 
\bea \label{eq:WnZnWass}
d_1(s(W_n),s(V_n)) \le \left(\frac{\theta}{\theta+1}\right)^nd_1(s(W_0),s(V_0)) \qm{for all $n \ge 0$.}
\ena  
\end{lemma}

\begin{proof}
By \eqref{eq:sw.sub.add} of Lemma \ref{lem:sto.order.pres},
the existence of $E[s(W_n)]$ implies the existence of $E[s(W_{n+1})]$. Now induction and the assumption that $E[s(W_n)]$ is finite for $n=0$ proves the expectation is finite for all $n \ge 0$.

The claim \eqref{eq:WnZnWass} holds trivially for $n=0$. Assuming it holds for some $n \ge 0$, let the joint distribution of $(s(V_n),s(W_n))$ achieve the infimum in \eqref{def:was.inf}. Then independently constructing $U_n \sim {\cal U}[0,1]$ on the same space as $s(V_n)$ and $s(W_n)$, the pair $s(W_{n+1}),s(V_{n+1})$ given by \eqref{eq:tVervaat.insec} are defined on the same space and have the desired marginals, and satisfy
\beas
|s(V_{n+1})-s(W_{n+1})| \le U_n^{1/\theta}|s(V_n)-s(W_n)|.
\enas
Hence, by the independence of $s(W_n)$ and $s(V_n)$ from $U_n$ and definition \eqref{def:was.inf} of the $d_1$ metric, we obtain
\beas
d_1(s(W_{n+1}),s(V_{n+1})) \le E[U_n^{1/\theta}|s(V_n)-s(W_n)|] = \frac{\theta}{\theta+1}E|s(V_n)-s(W_n)|,
\enas
and applying the induction hypotheses, we obtain \eqref{eq:WnZnWass}. 
\end{proof}

Define the generalized inverse of an increasing function $s:[0,\infty) \rightarrow [0,\infty)$ as 
\bea\label{geninv}
s^{-}(x)=\inf\{y : s(y) \ge x\}
\ena
with the convention that $\inf \emptyset= \infty$. In particular for $X$ a random variable, we consider $s^{-}(X)$ as a random variable taking values in the extended real line. When writing the stochastic order relation $V \le_{\rm st} W$ between two extended valued random variables, we mean that $P(V \ge t) \le P(W \ge t)$ holds for all $t$ in the extended real line. Note that $s^{-}(\cdot)$ and $s^{-1}(\cdot)$ coincide on the range of $s(\cdot)$ when $s(\cdot)$ is continuous and strictly increasing.

\begin{theorem}\label{existence}
Let $\theta>0$ and $s(\cdot)$ satisfy Condition \ref{cond:s}.  
Then there exists a unique distribution ${\cal D}_{\theta,s}$ for a random variable $D_{\theta,s}$ such that $s(D_{\theta,s})$ has finite mean and satisfies $D_{\theta,s}=_d D_{\theta,s}^*$, with $D_{\theta,s}^*$ given by \eqref{bias}. In addition, $D_{\theta,s} \le_{\rm st} s^{-}(D_\theta)$.
\end{theorem}

\begin{proof}
Generate a sequence $W_n, n \ge 0$ as in \eqref{eq:Wn.t.rec} with initial value $W_0=0$. 
We first prove that a distributional fixed point to the transformation \eqref{bias} exists by showing the existence of a distribution ${\cal D}_{\theta,s}$ and a subsequence $(n_k)_{k \ge 0}$ such that
\bea \label{eq:Wnnsubk}
W_{n_k} \rightarrow_d {\cal D}_{\theta,s} \qmq{and}  W_{n_k+1} \rightarrow _d {\cal D}_{\theta,s} \qmq{as $k \to \infty$ and} W_{n_k+1}=_d W_{n_k}^*.
\ena

By Lemma \ref{lem:sto.order.pres} and the fact that $s(W_0)=s(0)=0$, we have $s(W_n) \le_{\rm st} D_\theta$ for all $n \ge 0$. 
As $0 \le s(W_n) \le_{\rm st} D_\theta$, the sequence $s(W_n), n \ge 0$ is tight and therefore has a convergent subsequence $s(W_{n_k}) \rightarrow_d {\cal E}_{\theta,s}$ for some distribution ${\cal E}_{\theta,s}$. As $s(\cdot)$ is invertible $W_{n_k} \rightarrow_d {\cal D}_{\theta,s}$ where  ${\cal D}_{\theta,s}=_ds^{-1}({\cal E}_{\theta,s})$ proving first claim in \eqref{eq:Wnnsubk}. As weak limits preserve stochastic order, ${\cal E}_{\theta,s} \le_{\rm st} {\cal D}_\theta$ and hence ${\cal D}_{\theta,s} \le_{\rm st} s^{-}({\cal D}_\theta)$, as $s(\cdot)$ increasing implies that $s^{-}(\cdot)$ given by \eqref{geninv} is also increasing. The last claim of the theorem is shown.                                                                

Let the sequence $V_n, n \ge 0$ be generated as $W_n$ is in \eqref{eq:Wn.t.rec} with initial value $V_0=_d W_1$ and $V_0$ independent of $U_n, n  \ge 0$. Note that $s(V_0)$ has finite mean by Lemma \ref{lem:was.s(W)s(V)}, and hence \eqref{eq:WnZnWass} may be invoked to conclude that $d_1(s(W_n),s(V_n)) \rightarrow 0$ as $n \rightarrow \infty$. As $s(W_{n_k}) \rightarrow_d {\cal E}_{\theta,s}$, we have $s(V_{n_k}) \rightarrow_d {\cal E}_{\theta,s}$ hence $V_{n_k} \rightarrow_d {\cal D}_{\theta,s}$. As $V_0=_d W_1$, we have $V_n=_d W_{n+1}$, implying $W_{n_k+1} \rightarrow_d {\cal D}_{\theta,s}$. The second claim in \eqref{eq:Wnnsubk} is shown. The third claim holds by \eqref{eq:Wn.t.rec} and by definition \eqref{eq:Wn.t.rec*} of the Dickman bias transform.

By the first claim in \eqref{eq:Wnnsubk} and the continuity of $s(\cdot)$ and $s^{-1}(\cdot)$, letting $U \sim {\cal U}[0,1]$ be independent of $D_{\theta,s} \sim {\cal D}_{\theta, s}$ and $W_{n_k}$, as $k \rightarrow \infty$ we have
\beas
W_{n_k}^*=_ds^{-1}\left(U^{1/\theta}s(W_{n_k}+1) \right) \rightarrow_d s^{-1}\left(U^{1/\theta}s(D_{\theta,s}+1) \right)=_d D_{\theta,s}^*.
\enas
Hence, letting $k \rightarrow \infty$ in the third relation \eqref{eq:Wnnsubk} we obtain $D_{\theta,s}=D_{\theta,s}^*$, showing that ${\cal D}_{\theta,s}$  is a fixed point of the Dickman bias transformation \eqref{eq:Wn.t.rec*}.

Now let $W_0$ and $V_0$ be any two fixed points of the transformation such that $s(W_0)$ and $s(V_0)$ have finite mean. Then the distributions of $s(W_n)$ and $s(V_n)$ do not depend on $n$, and \eqref{eq:WnZnWass} yields
\beas
d_1(s(W_0),s(V_0))=d_1(s(W_n),s(V_n)) \rightarrow 0 \qm{as $n \rightarrow \infty$.}
\enas
Hence $s(W_0)=_d s(V_0)$, and applying $s^{-1}$ we conclude $W_0=_dV_0$; the fixed point is unique.
\end{proof}

\subsection{Distributional bounds for ${\cal D}_{\theta,s}$  approximation and Simulations}\label{bounds}

In this section we study the accuracy of recursive methods to approximately sample from the ${\cal D}_{\theta,s}$ family, starting with the following simple corollary to Lemma \ref{lem:was.s(W)s(V)} that gives a bound on how well the utility $s(W_n)$, satisfying the recursion \eqref{eq:tVervaat.insec}, approximates the long term utility of the fixed point.

\begin{corollary} \label{cor:dist.sW.sD}
	Let $\theta>0$ and  Condition \ref{cond:s} be in force. Then $s(W_n)$ given by \eqref{eq:tVervaat.insec} satisfies
	\beas
	d_1(s(W_n),s(D_{\theta,s})) \le \left(\frac{\theta}{\theta+1}\right)^n d_1(s(W_0),s(D_{\theta,s})) \qm{for all $n \ge 0$.}
	\enas
\end{corollary}

\begin{proof}
	The result follows from \eqref{eq:WnZnWass} of Lemma \ref{lem:was.s(W)s(V)} by taking $V_0=_d D_{\theta,s}$ and noting that $D_{\theta,s}$ is fixed by the transformation \eqref{eq:Wn.t.rec*} so that $s(V_n)=_d s(D_{\theta,s})$ for all $n$.
\end{proof}

Corollary \ref{cor:dist.sW.sD} depends on the direct coupling in Lemma \ref{lem:was.s(W)s(V)}, which constructs the variables $s(W_n)$ and $s(V_n)$ on the same space. Theorem \ref{bound} below gives a bound for when a non-negative random variable $W$ is used to approximate the distribution of $D_{ \theta,s}$. Though direct coupling can still be used to obtain bounds such as those in Theorem \ref{bound} for the ${\cal D}_\theta$ family, doing so is no longer possible for the more general ${\cal D}_{\theta,s}$ family as iterates of \eqref{eq:Wn.t.rec} can no longer be written explicitly when $s(\cdot)$ is non-linear. Theorem \ref{bound} below provides a Wasserstein bound between $D_{\theta,s}$ and $W$ assuming certain natural conditions on the function $s(\cdot)$.

For $\theta>0$, suppressed in the notation, and $x>0$ such that $s'(x)$ exists, let
\bea\label{I:def}
I(x)=\frac{\theta s'(x)}{s^{\theta+1}(x)} \int_0^x s^\theta(v) dv.
\ena

For $S \subset [0, \infty)$, we say a function $f:[0, \infty) \to [0, \infty)$ is locally absolutely continuous on $S$ if it is absolutely continuous when restricted to any compact sub-interval of $S$. Unless otherwise stated, locally absolutely continuity will mean over the domain of $f(\cdot)$.

\begin{theorem}\label{bound} Let $\theta>0$ and $s:[0,\infty) \rightarrow [0,\infty)$  satisfying Condition \ref{cond:s} be locally absolutely continuous on $[0,\infty)$ and such that $E[D_{\theta,s}] <\infty$. With $I(\cdot)$ as in \eqref{I:def}, if there exists $\rho \in [0,1)$ such that
\bea\label{rholess1}
\|I\|_\infty  \le \rho,
	\ena
then for any non-negative random variable $W$ with finite mean,
\bea \label{eq:thm2.2bd}
d_1(W,D_{\theta,s}) \le (1-\rho)^{-1} d_1(W^*,W).
\ena
In the special case $s(x)=x$, $\|I\|_\infty  = \theta/(\theta+1) \in [0,1)$, and one may take $\rho$ equal to this value.
\end{theorem}
\begin{remark}
Note that $E[s^{-1}(D_\theta)] < \infty$ implies $E[D_{\theta,s}]<\infty$ as $D_{\theta,s} \le_{\rm st} s^{-1}(D_\theta)$ by Theorem \ref{existence}.
\end{remark}

\begin{remark}
By a simple argument, similar to the one in Section 3 of \cite{Go17v2}, for $\theta>0$ and $s:[0,\infty) \rightarrow [0,\infty)$ satisfying Condition \ref{cond:s}, \eqref{ConvexBd} below and $E[D_{\theta,s}] <\infty$, for any non-negative random variable $W$ with finite mean, we have
		\beas
		d_1(W,D_{\theta,s}) \le (1+\theta) d_1(W^*,W)
		\enas
		so that \eqref{eq:thm2.2bd} holds with $\rho=\theta/(\theta+1)$. 
		
		 The use of Stein's method in Theorem \ref{bound} does not require that $s(\cdot)$ satisfy \eqref{ConvexBd} but does need $s(\cdot)$ to be locally absolutely continuous. In addition, the alternative approach in \cite{Go17v2} has no scope for improvement in terms of finding the best constant $\rho$; Example \ref{ex:liminf} presents a case where taking $\rho=\theta/(\theta+1)$ is not optimal. Theorem \ref{concaverho} below gives a verifiable criteria by which one can show when the canonical choice $\rho=\theta/(\theta+1)$ is not improvable.
\end{remark}

We will prove Theorem \ref{bound} using Stein's method in Section \ref{Smoothness}. Here, we provide the following corollary applicable for the simulation of ${\cal D}_{\theta,s}$ distributed random variables. Note that when $s(\cdot)$ is strictly increasing and continuous, for $W$ independent of $U \sim {\cal U}[0,1]$ the transform $W^*$ as given by \eqref{bias} satisfies
\bea\label{biassotcbd}
W^*=_d s^{-1} (U^{1/\theta}s(W+1)) \le W+1.
\ena

\begin{corollary}\label{othertwo}
	Let $s:[0,\infty) \rightarrow [0,\infty)$ be as in Theorem \ref{bound} and let $\{W_n, n\ge 1\}$ be generated by \eqref{eq:Wn.t.rec} with $W_0$ non-negative and $EW_0<\infty$, independent of $\{U_n, n \ge 0\}$. If $\rho\in [0,1)$ exists satisfying \eqref{rholess1}, then
	\bea\label{DistrApp}
	d_1(W_n, D_{\theta,s}) \le (1-\rho)^{-1} d_1(W_{n+1},W_n).
	\ena  
		Moreover, if $s(\cdot)$ satisfies
		\bea\label{ConvexBd}
		| s^{-1}(as(x)) -s^{-1}(as(y)) | \le a|x-y| \qmq{ for $a \in [0,1]$ and $x,y \ge 1$,}
		\ena
		then
	\bea\label{forremark}
	d_1(W_n,D_{\theta,s}) \le (1-\rho)^{-1} \left(\frac{\theta}{\theta+1}\right)^n d_1(W_1,W_0).
	\ena
	When $W_0=0$,
	\bea\label{GenSimBd}
	d_1(W_n, D_{\theta,s})
	\le(1-\rho)^{-1} \left(\frac{\theta}{\theta +1}\right)^n E[s^{-1}(U^{1/\theta})],
	\ena
	and in the particular the case of the generalized Dickman ${\cal D}_\theta$ family,
	\bea \label{not.GenSimBd}
	d_1(W_n,D_\theta) \le \theta \left(\frac{\theta}{\theta+1}\right)^{n}.
	\ena
\end{corollary}
\begin{proof}
Identity \eqref{eq:Wn.t.rec}, the inequality in \eqref{biassotcbd} and induction show that $W_n \le W_0+n$, and hence $EW_n <\infty$, for all $n \ge 0$. Inequality \eqref{DistrApp} now follows from Theorem \ref{bound} noting from \eqref{bias} that $W_n^*=_d W_{n+1}$ for all $n \ge 0$.
	
To show \eqref{forremark}, recalling that the bound \eqref{def:was.inf} is achieved for real valued random variables, for every $n \ge 1$ we may construct $W_{n-1}'$ and $V_n'$ independent of $U_n$ such that $W_{n-1}'=_d W_{n-1},V_n'=_d W_n$ and $E|V_n'-W_{n-1}'|=d_1(W_n,W_{n-1})$. 
	Now letting
	\beas
W_n''=s^{-1}(U_n^{1/\theta} s(W_{n-1}'+1)) \qmq{and} V_{n+1}''=s^{-1}(U_n^{1/\theta} s(V_n'+1))
	\enas 
	we have $W_n''=_d W_n$ and $V_{n+1}''=_d W_{n+1}$. Thus, using \eqref{def:was.inf} followed by \eqref{ConvexBd} we have
\begin{multline*}
d_1(W_{n+1},W_n) \le E|V_{n+1}''-W_n''| \\
= E|s^{-1}(U_n^{1/\theta} s(V_n' +1)) - s^{-1}(U_n^{1/\theta} s(W_{n-1}' +1))|\\
\le E[U_n^{1/\theta}|V_n'-W_{n-1}'|] = \frac{\theta}{\theta+1}d_1(W_n,W_{n-1}).
\end{multline*}
Induction now yields
\beas
d_1(W_{n+1},W_n) 
\le  \left(\frac{\theta}{\theta+1}\right)^n d_1(W_1,W_0)
\enas
and applying \eqref{DistrApp} we obtain \eqref{forremark}.

Inequality \eqref{GenSimBd} now follows from \eqref{forremark} noting in this case, using $s(1)=1$, that $(W_0,W_1)=(0,s^{-1}(U_0^{1/\theta}))$, and \eqref{not.GenSimBd} is now achieved from \eqref{GenSimBd} by taking
$\rho$ to be $\theta/(\theta+1)$, as provided by Theorem \ref{bound} when $s(x)=x$.
\end{proof}

In the remainder of this subsection, in Lemma \ref{lem:concave.suffices} we present some general and easily verifiable conditions on $s(\cdot)$ for the satisfaction of \eqref{ConvexBd}, and in Theorem \ref{concaverho} ones under which the integral bound 
$\|I\|_\infty  \le \rho$ in \eqref{rholess1} holds with $\rho \in [0,1)$. Lastly we show our bounds are equivalent to what can be obtained by a direct coupling method, in the cases where the latter is available.

\begin{condition} \label{cond:strong}
	The function $s:[0,\infty) \rightarrow [0,\infty)$ is continuous at $0$, strictly increasing with $s(0)=0$ and $s(1)=1$, and concave.
\end{condition}

\begin{lemma}\label{abs.cont.}
	If a function $f:[0, \infty) \to [0, \infty)$ is increasing, continuous at $0$ and locally absolutely continuous on $(0,\infty)$, then it is locally absolutely continuous on its domain.
\end{lemma}
\begin{proof}
  Since $f(\cdot)$ is absolutely continuous on any compact subset of $(0, \infty)$, by continuity of $f(\cdot)$ at $0$, for $0<\epsilon \le x <\infty$, using absolute continuity on $[\epsilon,x]$ in the second equality and monotone convergence in the third, we have
\beas
f(x)-f(0)=\lim_{\epsilon \downarrow 0}(f(x)-f(\epsilon))=\lim_{\epsilon \downarrow 0} \int_{\epsilon}^{x}f'(v)dv=\int_{0}^{x}f'(v)dv.
\enas
Hence $f(\cdot)$ is locally absolutely continuous on its domain.
\end{proof}
\begin{lemma} \label{lem:concave.suffices}
	If $s:[0,\infty) \rightarrow [0,\infty)$ satisfies Condition \ref{cond:strong}, then it is locally absolutely continuous on $[0,\infty)$, satisfies Condition \ref{cond:s} and 
	\bea \label{eq:s.inverse.inequality}
	|s^{-1}(as(y)) - s^{-1}(as(x))| \le a |y-x| \qm{for all $x,y \ge 0$ and $a \in [0,1]$.}
	\ena 
\end{lemma}
\begin{proof}
	First, since $s(\cdot)$ is concave, it is locally absolutely continuous on $(0, \infty)$. Thus, by Lemma \ref{abs.cont.}, $s(\cdot)$ is locally absolutely continuous on its domain. Next we show $s(\cdot)$ is subadditive, that is, that
	\bea \label{eq:s.sub}
	s(x+y) \le s(x)+s(y) \qm{for $x,y \ge 0$.}
	\ena
	Taking $x,y \ge 0$, we may assume both $x$ and $y$ are non-zero as \eqref{eq:s.sub} is trivial otherwise since $s(0)=0$.  By concavity,
	\beas
	\frac{y}{x+y} s(0) + \frac{x}{x+y}s(x+y) \le s(x) \qmq{and}  \frac{x}{x+y} s(0) + \frac{y}{x+y}s(x+y) \le s(y).
	\enas

	Since $s(0) = 0$, adding these two inequalities yield \eqref{eq:s.sub}.
	Taking $y=1$ and using $s(1)=1$ we obtain \eqref{Subadditivity}.  Next, the local absolute continuity and concavity of $s(\cdot)$ on $[0,\infty)$ imply that it is almost everywhere differentiable on this domain, with $s'(\cdot)$ decreasing almost everywhere. Thus for $x \ge y \ge 0$, we have
	\beas
	s(x+1)-s(x)=\int_x^{x+1} s'(u)du \le \int_x^{x+1}s'(u+y-x)du = \int_y^{y+1} s'(u) du=s(y+1)-s(y),
	\enas
	which together with the fact that $s(\cdot)$ is increasing implies \eqref{Concavity}. Hence $s(\cdot)$ satisfies Condition \ref{cond:s}.

	Lastly, we show that $s(\cdot)$ satisfies \eqref{eq:s.inverse.inequality}.
	Since $s(0)=0$ the inequality is trivially satisfied for $a=0$, so fix some $a \in (0,1]$. Again as the result is trivial otherwise, we may take $x \not = y$; without loss, let $0 \le x <y$.  The inverse function  $r(\cdot)=s^{-1}(\cdot)$ is continuous at zero and convex on the range $S$ of $s(\cdot)$, a possibly unbounded convex subset $[0,\infty)$ that 
	includes the origin. Letting $u=s(x)$ and $v=s(y)$, as $s(\cdot)$, and hence $r(\cdot)$, are strictly increasing and $x \not = y$, inequality \eqref{eq:s.inverse.inequality} may be written 
	\bea \label{eq:secants}
	r(av)-r(au) \le a(r(v)-r(u)) \qmq{or equivalently} \frac{r(av)-r(au)}{av-au} \le \frac{r(v)-r(u)}{v-u},
	\ena
	where all arguments of $r(\cdot)$ in \eqref{eq:secants} lie in $S$, it being a convex set containing $\{0,u,v\}$.
	
	The second inequality in \eqref{eq:secants} follows from the following slightly more general one that any convex function $r :[0, \infty)\to [0, \infty)$ which is continuous at $0$  satisfies by virtue of its local absolute continuity and a.e. derivative $r'(\cdot)$ being increasing: if $(u_1,v_1)$ and $(u_2,v_2)$ are such that $u_1 \not = v_1$, $u_1 \le u_2$ and $v_1 \le v_2$, and all these values lie in the range of $r(\cdot)$, then
	\begin{multline*}
	\frac{r(v_1)-r(u_1)}{v_1-u_1} = \frac{1}{v_1-u_1}\int_{u_1}^{v_1} r'(w)dw=
	\int_0^1 r'(u_1+(v_1-u_1)w)dw\\
	\le \int_0^1 r'(u_2+(v_2-u_2)w)dw = \frac{1}{v_2-u_2}\int_{u_2}^{v_2}r'(w)dw = \frac{r(v_2)-r(u_2)}{v_2-u_2},
	\end{multline*}
	as one easily has that $u_1+(v_1-u_1)w\le u_2+(v_2-u_2)w$ for all $w \in [0,1]$.
\end{proof}

When the function $s(\cdot)$ is nice enough, we can actually say more about the constant $\rho$ in \eqref{rholess1} of Theorem \ref{bound}.
\begin{theorem}\label{concaverho}
	Assume that $\theta>0$ and $s:[0,\infty) \rightarrow [0,\infty)$ is concave and continuous at $0$. Then with $I(x)$ as given in \eqref{I:def},
	\bea\label{I:Bound}
     \|I\|_\infty \le \frac{\theta}{\theta +1}.
	\ena
     If moreover $s(\cdot)$ is strictly increasing  with $s(0)=0$ and $\lim_{n \to \infty} s'(x_n)<\infty$ for some sequence of distinct real numbers $x_n \downarrow 0$ in the domain of $s'(\cdot)$, then
	\bea\label{I:BoundSpec}
	 \|I\|_\infty =\frac{\theta}{\theta +1}.
	\ena
	\end{theorem}

\begin{proof}
	Since $s(\cdot)$ is concave and continuous at $0$, it is locally absolutely continuous with $s'(\cdot)$ decreasing almost everywhere on $[0,\infty)$. Since $u^{\theta+1}$ is Lipschitz on any compact interval, by composition, $s^{\theta+1}(\cdot)$ is absolutely continuous on $[0,x]$ for any $x\ge 0$, and thus  for almost every $x$,
	\begin{multline*}
	\frac{(\theta+1)I(x)}{\theta}= \frac{(\theta+1) s'(x)}{s^{\theta+1}(x)} \int_0^x s^\theta(v) dv
	\le \frac{1}{s^{\theta+1}(x)} \int_0^x (\theta+1) s^{\theta}(v)s'(v) dv\\
	=\frac{s^{\theta+1} (x)-s^{\theta+1} (0)}{s^{\theta+1} (x)} \le 1,
	\end{multline*}
	 proving \eqref{I:Bound}.

	To prove the second claim, first note that $0<\lim_{n \to \infty} s'(x_n)<\infty$, the existence of the limit and second inequality holding by assumption, and the first inequality holding as $s(\cdot)$ is strictly increasing and $s'(\cdot)$ is decreasing almost everywhere.
	
Thus, in the second equality using a version of the Stolz-Ces{\`a}ro theorem \cite{St885} adapted to accommodate $s^{\theta+1}(x_n)$ decreasing to zero,
	\begin{multline*}
	\lim_{n \to \infty} I(x_n)=\theta \lim_{n \to \infty} s'(x_n) \lim_{n \to \infty} \frac {\int_0^{x_n} s^\theta(v)dv}{s^{\theta+1}(x_n)}=\theta \lim_{n \to \infty} s'(x_n) \lim_{n \to \infty} \frac {\int_{x_{n+1}}^{x_n} s^\theta(v)dv}{s^{\theta+1}(x_n)-s^{\theta+1}(x_{n+1})}\\
	=\theta \lim_{n \to \infty} s'(x_n) \lim_{n \to \infty} \frac {\int_{x_{n+1}}^{x_n} s^\theta(v)dv}{(\theta+1)\int_{x_{n+1}}^{x_n} s^\theta(v)s'(v)dv}=\frac{\theta}{\theta+1}\lim_{n \to \infty} s'(x_n)\lim_{n \to \infty} \frac{1}{s'(x_n)}=\frac{\theta}{\theta+1},
	\end{multline*}
	where the penultimate equality follows from the fact that
	\beas
\lim_{n \to \infty} \frac{1}{s'(x_n)}=\lim_{n \to \infty} \frac{1}{s'(x_{n+1})}\le	\lim_{n \to \infty}\frac {\int_{x_{n+1}}^{x_n} s^\theta(v)dv}{\int_{x_{n+1}}^{x_n} s^\theta(v)s'(v)dv} \le \lim_{n \to \infty} \frac{1}{s'(x_n)}
	\enas
	and hence
	\beas
	\|I\|_\infty \ge \frac{\theta}{\theta+1}
	\enas
	which together with \eqref{I:Bound} proves \eqref{I:BoundSpec}.
	\end{proof}

The bound \eqref{not.GenSimBd} of Corollary \ref{othertwo} is obtained by specializing results for the ${\cal D}_{\theta,s}$ family, proven using the tools of Stein's method, to the case where $s(x)=x$. For this special case, letting $V_j=U_j^{1/\theta}$ for $j \ge 0$, the iterates of the recursion \eqref{eq:Wn.t.rec}, starting at $W_0=0$, can be written explicitly as
\beas
W_n = \sum_{k=0}^{n-1} \prod_{j=k}^{n-1}V_j,
\enas
allowing one to obtain bounds using direct coupling. Interestingly, the results obtained by both methods agree, as seen as follows. First, we show
\beas
W_n=_d Y_n \qmq{where} Y_n=\sum_{k=0}^{n-1} \prod_{j=0}^kV_j, \qmq{and} Y_\infty \sim D_\theta \qmq{where} Y_\infty=\sum_{k=0}^\infty \prod_{j=0}^kV_j.
\enas
The first claim is true since for every $n \ge 1$,
\beas
(V_0,\ldots,V_{n-1}) =_d (V_{n-1},\ldots,V_0).
\enas
For the second claim, note that the limit $Y_\infty$ exists almost everywhere and has finite mean by monotone convergence. Now using definition \eqref{def:W*}, with $U_{-1} \sim {\cal U}[0,1]$ independent of $U_0,U_1\ldots$ and setting $V_{-1}=U_{-1}^{1/\theta}$, we have
\begin{multline*}
Y_\infty^* = U_{-1}^{1/\theta}(Y_\infty+1) = V_{-1}\left(\sum_{k=0}^\infty \prod_{j=0}^kV_j+1 \right) \\
= \sum_{k=0}^\infty \prod_{j=-1}^kV_j + V_{-1}
=\sum_{k=-1}^\infty \prod_{j=-1}^kV_j = \sum_{k=0}^\infty \prod_{j=0}^kV_{j-1} = _d \sum_{k=0}^\infty \prod_{j=0}^kV_j=Y_\infty. 
\end{multline*}
Hence $Y_\infty \sim {\cal D}_\theta$.
As $(Y_n,Y_\infty)$ is a coupling of a variable with the $W_n$ distribution to one with the ${\cal D}_\theta$ distribution, by \eqref{def:was.inf} we obtain
\begin{multline*}
	d_1(W_n, D_\theta) = d_1(Y_n, Y_\infty) \le E|Y_\infty-Y_n| = E\left(\sum_{k=n}^\infty \prod_{j=0}^k V_j\right)\\= \sum_{k=n}^\infty \left(\frac{\theta}{\theta+1}\right)^{k+1}= \theta \left(\frac{\theta}{\theta+1}\right)^n,
\end{multline*}
in agreement with \eqref{not.GenSimBd}.

\subsection{Examples}\label{subsec:dickman.examples}
We now consider three new distributions that arise as special cases of the ${\cal D}_{\theta,s}$ family. Expected Utility (EU) theory has long been considered as an acceptable paradigm for decision making under uncertainty by researchers in both economics and finance, see e.g. \cite{ELS05}. To obtain tractable solutions to many problems in economics, one often restricts the EU criterion to a certain class of utility functions, which includes in particular the ones in Examples \ref{ex:s.is.exp} and \ref{ex:s.is.log}. In these two examples we apply the bounds provided in Corollary \ref{othertwo} for the simulation of the limiting distributions these functions give rise to via the recursion \eqref{eq:Wn.t.rec} with say, $W_0=0$. For each example we will verify Condition \ref{cond:strong}, implying Condition \ref{cond:s} by Lemma \ref{lem:concave.suffices}, and hence existence and uniqueness of $D_{\theta,s}$.
\begin{example}\label{ex:s.is.exp}
The exponential utility function $u(x)=1-e^{-\alpha x}$ is the only model, up to linear transformations, exhibiting constant absolute risk aversion (CARA), see \cite{ELS05}. Since utility is unique up to linear transformations, we consider its scaled version
	\beas
	s_{\alpha}(x)=\frac{1-e^{-\alpha x}}{1-e^{-\alpha}} \qmq{for} x \ge 0
	\enas
	characterized by a parameter $\alpha>0$. Clearly $s_\alpha(\cdot)$ is continuous at $0$, strictly increasing with $s_\alpha(0)=0$ and $s_\alpha(1)=1$ and  concave. Since $\lim_{x \downarrow 0}s_\alpha'(x)= \alpha (1-e^{-\alpha})^{-1} \in (0,\infty)$, for all $\theta>0$, by \eqref{I:BoundSpec} of Theorem \ref{concaverho}, one can take $\rho$ to be $\theta/(\theta+1)$ and not strictly smaller, and \eqref{GenSimBd} of Corollary \ref{othertwo} yields
	\beas
	d_1(W_n, D_{\theta,s_\alpha}) \le \theta \left(\frac{\theta}{\theta +1}\right)^{n-1} \qm{for all $n \ge 0$,}
	\enas
	using that $0 \le s_\alpha^{-1}(U^{1/\theta})\le s_\alpha^{-1}(1)=1$ almost surely. 
	
Letting $W_\alpha \sim D_{\theta,s_\alpha}$ it is easy to verify that
\beas
s_\alpha(W_\alpha) =_d U^{1/\theta} s_\alpha(W_\alpha +1) = U^{1/\theta}(1 + e^{-\alpha}s_\alpha(W_\alpha)).
\enas
Using this identity, that Theorem \ref{existence} gives $0 \le s_\alpha(W_\alpha) \le_{\mathrm {st}}
D_\theta$ for all $\alpha>0$, and that $\lim_{\alpha \downarrow 0}s_\alpha(x)=x$ for all $ x \ge 0$ one can show that $W_\alpha$ converges to $D_\theta$ as $\alpha \downarrow 0$. Hence, now setting $s_0(x)=x$, the family of models $D_{\theta,s_\alpha}, \alpha \ge 0$ is parameterized by a tuneable values of $\alpha \ge 0$ whose value may be chosen depending on a desired level of risk aversion, including the canonical $\alpha=0$ case where utility is linear.
\end{example}

\begin{example}\label{ex:liminf} Here we show how standard Vervaat perpetuity models can be seen to assume an implicit concave utility function, and how uncertainty in these utilities can be accommodated using the new families we introduce. Indeed, letting $\theta=1$ in \eqref{eq:tVervaat} and then $s_\theta (x)=x^\theta, \theta\in(0,1]$, it is easy to see that ${\cal D}_{1,s_\theta}={\cal D}_{\theta}$. To model situations where these utilities are themselves subject to uncertainty, we may let $A$ be a random variable supported in $(0,1]$ and consider the mixture $s(x)=E[s_A(x)]$.
	
More formally, for some $0<a \le 1$, let $\mu$ be a probability measure on the interval $(0,a]$, and define
	\beas
	s(x)=\int_0^a s_\alpha(x) d\mu(\alpha).
	\enas
	Since $0<a \le 1$, each $s_\alpha(\cdot)$ is concave and satisfies Condition \ref{cond:strong} and hence so does $s(\cdot)$. By \eqref
	{I:Bound} of Theorem \ref{concaverho}, for the family ${\cal
		D}_{\theta,s}$ one can take $\rho=\theta/(\theta+1)$.
	
	Fix $l >0$.  For $x \ge l$, note that $\partial x^\alpha/\partial x= \alpha x^{\alpha-1} \le\alpha l^{\alpha-1}$ which is bounded and hence $\mu$-integrable on $[0,a]$. Thus by dominated convergence, since $l>0$ is arbitrary, we obtain
	\bea\label{DiffLem}
	s'(x)=\int_0^a \frac{\partial x^\alpha}{\partial x} d\mu(\alpha)=\int_{0}^a \alpha x^{\alpha-1}d\mu(\alpha) \qm{for all $x>0$.}
	\ena
Now note that for $a<1$, $\lim_{x \downarrow 0} s'(x)$ diverges to infinity, and hence \eqref{I:BoundSpec} of Theorem \ref{concaverho} cannot be invoked. We show, in fact, that one may obtain a bound better than $\theta/(\theta+1)$ in this case.

Taking $\theta=1$ and computing $I(x)$ directly from \eqref{I:def}, using \eqref{DiffLem} for the first equality and Fubini's theorem for the second, we have
	\begin{multline*}
	I(x)=\frac{\left[\int_{0}^a \alpha x^{\alpha-1}d\mu(\alpha)\right] \left[\int_{0}^x \int_0^a v^\alpha d\mu(\alpha) dv\right]}{\left[\int_{0}^a x^{\alpha} d\mu(\alpha)\right]^2}
	=\frac{\left[\int_{0}^a \alpha x^{\alpha-1}d\mu(\alpha)\right] \left[\int_{0}^a \frac{x^{\alpha+1}}{\alpha+1} d\mu(\alpha) \right]}{\left[\int_{0}^a x^{\alpha} d\mu(\alpha)\right]^2}\\
	=\frac{\left[\int_{0}^a \int_0^a \frac{\alpha}{\beta+1} x^{\alpha + \beta}d\mu(\alpha) d\mu(\beta)\right]}{\left[\int_{0}^a x^{\alpha} d\mu(\alpha)\right]^2}=\frac{\left[\int_{0}^a \int_0^a \frac{1}{2}\left(\frac{\alpha}{\beta+1}+ \frac{\beta}{\alpha+1}\right) x^{\alpha + \beta}d\mu(\alpha) d\mu(\beta)\right]}{\int_{0}^a \int_0^a x^{\alpha+\beta} d\mu(\alpha)d\mu(\beta)}\\
	\le \sup_{\alpha, \beta \in [0,a]}\frac{1}{2}\left(\frac{\alpha}{\beta+1}+ \frac{\beta}{\alpha+1}\right).
	\end{multline*}
Taking $0 \le \alpha \le \beta \le a$, the reverse case being handled similarly, using the simple fact that
	\begin{eqnarray*}
		(\beta-\alpha)^2 \le\beta-\alpha\qquad\qmq{for $0 \le\alpha\le\beta \le1$}
	\end{eqnarray*}
	shows that for $0 \le\alpha\le\beta\le a$,
	\beas
	\frac{\alpha}{\beta+1}+\frac{\beta}{\alpha+1} \le \frac{2\beta}{\beta+1} \le \frac{2a}{a+1}
	\enas
	and hence one can take $\rho=a/(a+1)$. Note that when $a=1/2$, say, we obtain the upper bound $\rho=1/3$, whereas the bound \eqref{I:Bound} of Theorem \ref{concaverho} gives $1/2$ when $\theta=1$.Taking $\mu$ to be unit mass at  $1$ yields $\rho=1/2$ which recovers the bound on $\rho$ for the standard Dickman derived in \cite{Go17v2}, and as given in Theorem \ref{bound}, for the value $\theta=1$.
\end{example}

\begin{example} \label{ex:s.is.log}
The logarithm $u(x)=\log x$ is another commonly used utility function as it exhibits constant relative risk aversion (CRRA) which often simplifies many problems encountered in macroeconomics and finance, see \cite{ELS05}. Applying a shift to make it non-negative, let
\beas
s(x)=\log(x+1)/ \log 2 \qm{for $x \ge 0$.}
\enas
Clearly $s(\cdot)$ satisfies Condition \ref{cond:strong}. To apply Corollary \ref{othertwo} it remains to compute an upper bound $\rho$ on the integral in \eqref{I:def}. Now since $\lim_{x \downarrow 0}s'(x) < \infty $, by \eqref{I:BoundSpec} of Theorem \ref{concaverho}, we may take
$\rho={\theta}/{(\theta+1)}$.
Noting
\beas
s^{-1}(x)=2^x-1,
\enas
simulating from this distribution by the recursion
\beas
W_{n+1} = (W_n+2)^{U_n^{1/\theta}} - 1 \qmq{for $n \ge 1$ with initial value $W_0=0$,}
\enas
inequality \eqref{GenSimBd} of Corollary \ref{othertwo} yields
\beas
d_1(W_n, D_{\theta,s}) \le \theta \left(\frac{\theta}{\theta +1}\right)^{n-1} \qm{for all $n \ge 0$,}
\enas
using that $0 \le s^{-1}(U^{1/\theta})=2^{U^{1/\theta}}-1 \le 1$ almost surely. 
\end{example}

\section{Smoothness Bounds}\label{Smoothness}
 In this section we turn to proving Theorem  \ref{Lipschitzg} from which Theorem \ref{bound} readily follows. We develop the necessary tools building on \cite{Go17v1}. For notational simplicity, in this section given $(\theta,s)$, let
\bea \label{eq:t=s^theta}
t(x)=s^\theta(x) \qm{for all $x \ge 0$.}
\ena
Throughout this section $t:[0,\infty) \to [0,\infty)$ will be strictly increasing and hence almost everywhere differentiable by Lebesgue's Theorem, see e.g. Section 6.2 of \cite{RF}, inducing the measure $\nu$ satisfying $d\nu/dv=t'(v)$ on $[0,\infty)$, where $v$ is Lebesgue measure. For $h \in L^1([0,a],\nu)$ for some $a>0$, define the averaging operator
\bea\label{avg}
A_x h=\frac{1}{t(x)}\int_0^x h(v)t'(v)dv  \qmq{for $x \in (0,a]$, and} A_0h=h(0) \mathds{1}(t(0)=0).
\ena

\begin{lemma}\label{lem:condf}
Let $t:[0,\infty) \to [0,\infty)$ be a strictly increasing function. If $h \in L^1([0,a], \nu)$ for some $a>0$, then
\bea\label{avgeqv}
f(x)=A_x h \qmq{satisfies} \frac{t(x)}{t'(x)} f'(x)+f(x)=h(x) \qmq{a.e. on $(0,a]$.}
\ena
Conversely, if in addition $t(\cdot)$ is locally absolutely continuous on $[0,\infty)$ with $t(0)=0$, and $f\in \bigcup_{\alpha \ge 0} {\rm Lip}_\alpha$, then the function $h(\cdot)$ as given by the right hand side of \eqref{avgeqv} is in $L^1([0,a], \nu)$ for all $a>0$ and
\bea\label{lem4.1:2}
f(x)=A_x h \qmq{for all $x \in (0,\infty)$.}
\ena
\end{lemma}
    \begin{proof}
    	The first claim follows from the definition \eqref{avg} of $A_xh$ by differentiation. For the second claim, noting that the case $\alpha=0$ is trivial, fix $\alpha>0$. Since $t(\cdot)$ is locally absolutely continuous and increasing, for any $a>0$,
    	\beas
    	\Bvert\int_0^a h(v)t'(v)dv\Bvert \le \int_0^a \left(t(v)|f'(v)| + |f(v)| t'(v) \right) dv \le \alpha a t(a) + (|f(0)|+\alpha a)t(a) <\infty
    	\enas
    	and hence $h \in L^1([0,a], \nu)$ for all $a>0$. Now note that the function $f(x)t(x)$ is locally absolutely continuous on $[0,\infty)$ since both $f(\cdot)$ and $t(\cdot)$ are locally absolutely continuous and for any compact $C \subset (0,\infty)$, the function $g(u,v)=uv$ is Lipschitz on $f(C) \times t(C)$. Thus, for $x>0$, we have
    	\beas
    	A_x h =\frac{1}{t(x)}  \int_0^x h(v)t'(v)dv=\frac{1}{t(x)}\int_0^x(t(v)f'(v)+t'(v)f(v))dv=\frac{1}{t(x)}(f(x)t(x))=f(x).
       \enas
	\end{proof}
\begin{lemma}\label{biasrel}
Let $t:[0,\infty) \to [0,\infty)$ be given by $t^{1/\theta}(\cdot)=s(\cdot)$ for $s(\cdot)$ a strictly increasing locally absolutely continuous function on $[0,\infty)$ with $s(0)=0$. Then $t(\cdot)$ is also locally absolutely continuous on $[0,\infty)$. Moreover, for $W$ a non-negative random variable and $W^*$ with distribution as in \eqref{bias}, for $h \in \bigcap_{a \in S}L^1([0,a],\nu)$ where $S$ is the support of $W+1$,
	\bea \label{eq:pres.mean.zero}
	E[h(W^*)]=E[A_{W+1}h]
	\ena
	whenever either expectation above exists, and 
	letting $f(x)=A_{x}h$ for all $x \in S$, 
	\bea \label{eq:Eop*}
	E\left[\frac{t(W^*)}{t'(W^*)} f'(W^*)+f(W^*)\right]=E[f(W+1)],
	\ena
 when the expectation of either side exists.
\end{lemma}
\begin{proof}
Since $s(\cdot)$ is locally absolutely continuous on $[0,\infty)$ and the function $u^\theta$ is Lipschitz on any compact subset of $(0, \infty)$, we have that $t(\cdot)$ is locally absolutely continuous on $(0,\infty)$, and hence the first claim of the lemma follows by Lemma \ref{abs.cont.}.

Next, as $A_x h$ exists for all $x \in S$ for any $h(\cdot)$ satisfying the hypotheses of the lemma and $W^* \le_{\rm st}W+1$ by \eqref{biassotcbd}, the averages $A_{W+1}h$ and $A_{W^*} h$ both exist. Now let the expectation on the left hand side of \eqref{eq:pres.mean.zero} exist. Using \eqref{bias} and \eqref{eq:t=s^theta} for the first equality and applying the change of variable $v=ut(W+1)$ in the resulting integral, we obtain
	\begin{multline*}
	E[h(W^*)]=E[h(t^{-1}[U t(W+1)])]=E\int_0^1 h(t^{-1}[u t(W+1)]) du\\=E\left[\frac{1}{t(W+1)} \int_0^{t(W+1)} h(t^{-1}(v)) dv\right] = E\left[\frac{1}{t(W+1)}\int_0^{W+1} h(w)t'(w) dw \right]=E[A_{W+1} h],
	\end{multline*}
where in the second to last equality we have applied the change of variable $t(w)=v$ and the fact that $t(0)=0$. When the expectation on the right hand side of \eqref{eq:pres.mean.zero} exists we apply the same argument, reading the display above from right to left.

	To prove the second claim of the lemma, by an argument similar to the one at the start of Section 3 of \cite{Go17v1}, the distribution of $U^{1/\theta} s(W+1)$ is absolutely continuous with respect to Lebesgue measure, with density, say $p(\cdot)$. By a simple change of variable, we obtain that $W^*$ has density
	\beas
	p_{W^*}(x)= p(s(x))s'(x) \qm{almost everywhere,}
	\enas  
	and hence the distribution of $W^*$ is also absolutely continuous with respect to Lebesgue measure. Thus by \eqref{avgeqv},
	\beas
	E\left[\frac{t(W^*)}{t'(W^*)} f'(W^*)+f(W^*)\right]=E[h(W^*)]
	\enas
	and \eqref{eq:Eop*} follows from the first claim.
\end{proof}
For an a.e. differentiable function $f(\cdot)$, let
\bea\label{D}
\mathbb{D}_t  f(x)= \frac{t(x)}{t'(x)} f'(x)+ f(x)-f(x+1).
\ena

Note that if $f(x)=A_x g$ for some $g(\cdot)$, then under the conditions of Lemma \ref{lem:condf}, by \eqref{avgeqv} we may write \eqref{D} as
\bea\label{eq:stein.Ax+1.form}
\mathbb{D}_t f(x)=g(x)-A_{x+1}g \qmq{almost everywhere.}
\ena

Condition \ref{cond:s} is assumed in some of the following statements to assure that the distribution of $D_{\theta,s}$ exists uniquely. The proof of the next lemma is omitted, as it follows using Lemmas \ref{lem:condf} and \ref{biasrel}, similar to the proof of Lemma 3.2 in \cite{Go17v1}.

\begin{lemma} \label{lem:Dexpects} 
	Let $\theta>0$ and $s(\cdot)$ satisfy Condition \ref{cond:s}. If $s(\cdot)$ is locally absolutely continuous on $[0,\infty)$, then, 
	\beas
	E[h(D_{\theta,s})]=E[A_{D_{\theta,s}+1} h] \qmq{and}	E[\mathbb{D}_t f(D_{\theta,s})]=0,
	\enas
for all $h(\cdot) \in \bigcap_{a \in (0,\infty)}L^1([0,a],\nu)$ and $f(\cdot) \in \bigcup_{\alpha \ge 0} {\rm Lip}_\alpha$ for which $E[\mathbb{D}_t f(D_{\theta,s})]$ exists, respectively.
\end{lemma}

The second claim of the lemma and \eqref{D}  suggest the Stein equation 
 \bea \label{eq:f.stein}
 \frac{t(x)}{t'(x)}f'(x)+f(x)-f(x+1)=h(x)-E[h(D_{\theta,s})],
 \ena
which via \eqref{eq:stein.Ax+1.form} may be rewritten as
\bea\label{steineq}
g(x)-A_{x+1} g=h(x)-E[h(D_{\theta,s})]
\ena
whenever $g(\cdot)$ is such that $A_x g$ exists for all $x$ and $f(x)=A_x g$.

To prove Theorem \ref{bound}, we first need to identify a set of broad sufficient conditions on $t(\cdot)$ under which we can find a nice solution $g(\cdot)$ to \eqref{steineq} when $h \in {\rm Lip}_{1,0}$, where, suppressing dependence on $\theta$ and $s(\cdot)$ for notational simplicity, for $\alpha>0$, we let
\bea\label{Lip0}
{\rm Lip}_{\alpha,0}=\{h:[0,\infty) \to \mathbb{R}: h \in {\rm Lip}_\alpha, E[h(D_{\theta,s})]=0\}.
\ena

We note that the integral $I(x)$ in \eqref{I:def} can be written as the one appearing in \eqref{eq:18.for.t} below when $t(x)=s^\theta(x)$ as in \eqref{eq:t=s^theta}. Also note that by Lemma \ref{biasrel}, if $s(\cdot)$ is strictly increasing with $s(0)=0$, locally absolutely continuity of one of $s(\cdot)$ and $t(\cdot)$ implies that of the other. Hence, given that either one is locally absolutely continuous on $[0,\infty)$, as any continuous function $h:[0,\infty) \to \mathbb{R}$ is bounded on $[0,a]$
for all $a \ge 0$, we have $h \in \cap_{a>0} L^1([0,a],\nu)$. As the integrability of $h(\cdot)$ can thus be easily verified, it will not be given further mention.
\begin{lemma} \label{lem:t.rho.integral.bound}
Let $t:[0,\infty) \to [0,\infty)$ be a strictly increasing and locally absolutely continuous function on $[0,\infty)$. If $h(\cdot)$ is absolutely continuous on $[0,a]$ for some $a>0$ with a.e. derivative $h'(\cdot)$, then with $A_xh$ as in \eqref{avg}, 
\bea \label{eq:Axh.prime.general}
(A_x h)'=\frac{t'(x)}{t^2(x)} \int_0^x h'(u)t(u)du \qm{a.e.\ on $x \in (0,a]$.}
\ena
If there exists some $\rho \in [0,\infty)$ such that 
\bea \label{eq:18.for.t}
\esssup_{x>0}I(x) \le \rho \qmq{where} I(x)=\frac{t'(x)}{t^2(x)}\int_0^x  t(u) du,
\ena
then $A_x h \in {\rm Lip}_{\alpha\rho}$ on $[0,\infty)$ whenever $h \in {\rm Lip}_\alpha$ for some $\alpha \ge 0$.
\end{lemma}
\begin{proof}
For the first claim, first assume $h(0)=0$. Using Fubini's theorem in the third equality and then the local absolute continuity of $t(\cdot)$, for $x \in (0,a]$, we obtain
\begin{multline} \label{eq:Axh.Fubini}
A_x h=\frac{1}{t(x)}\int_0^x h(v)t'(v)dv=\frac{1}{t(x)}\int_0^x \int_0^v  t'(v)  h'(u)du dv\\
=\frac{1}{t(x)}\int_0^x \int_u^x  t'(v)  h'(u)dv du=\frac{1}{t(x)}\int_0^x h'(u) [t(x)-t(u)]du,
\end{multline}
and differentiation yields \eqref{eq:Axh.prime.general}. 

To handle the case where $h(0)$ is not necessarily equal to zero, letting  $h_0(x)=h(x)-h(0)$ the result follows by noting that $h_0'(\cdot)=h'(\cdot)$ and, by the absolute continuity of $t(\cdot)$, that $(A_x h_0)'=(A_x h - h(0))'=(A_xh)'$.

For the final claim,  using \eqref{eq:Axh.prime.general} and \eqref{eq:18.for.t}, for every $x$ for which $I(x) \le \rho$ and $t'(x)$ exists, we obtain
	\bea \label{eq:Axh'.alm.all.x}
	|(A_x h)'|=\Bvert\frac{t'(x)}{t^2(x)} \int_0^x h'(u)t(u)du\Bvert \le \|h'\|_\infty \frac{t'(x)}{t^2(x)} \int_0^x t(u)du \le\alpha \rho.
	\ena
	As $t(\cdot)$ is locally absolutely continuous, $A_xh$, as seen by the first equality in \eqref{eq:Axh.Fubini}, is a ratio of two locally absolutely continuous functions. For any fixed compact subset $C$ of $(0, \infty)$, since $u(x):=\int_0^x h(v)t'(v)dv$ is continuous, $u(C)$ is also compact and hence bounded. Also, since $t(\cdot)$ is strictly increasing with $t(0) \ge 0$, $t(C)$ is bounded away from $0$. Hence the function $f(u,v)=u/v$ restricted to $u(C)\times t(C)$ is Lipschitz, implying that $A_x h$ is absolutely continuous on $C$. Thus, it follows that $A_x h \in {\rm Lip}_{\alpha \rho}$, as only $x$ values in a set of measure zero have been excluded in \eqref{eq:Axh'.alm.all.x}.
\end{proof}

\begin{remark}
	If $\theta>0$ and $t$ is given by $t(\cdot)=s^{\theta}(\cdot)$ for $s(\cdot)$ concave and continuous at zero, then $\|I\|_\infty \le \theta/(\theta+1)$ by Theorem \ref{concaverho}. Hence $\rho \in [0,1)$ always exists for such choices of $t$.
\end{remark}

Lemmas \ref{lem:suph0}, \ref{lem:alpha.rho^n} and Theorem \ref{Lipschitzg} generalize Lemmas 3.5, 3.6 and Theorem 3.1 in \cite{Go17v1} for the generalized Dickman; their proofs follow closely those in \cite{Go17v1} and hence are omitted.

\begin{lemma} \label{lem:suph0}
	Let $\theta>0$ and $s(\cdot)$ satisfy Condition \ref{cond:s}. Moreover assume that $\mu=E[D_{\theta,s}]$ exists. Then with ${\rm Lip}_{\alpha,0}$ as in \eqref{Lip0}, for any $\alpha>0$,
	\bea \label{eq:sup.h0.gen.D}
	{\rm sup}_{h \in {\rm Lip}_{\alpha,0}} |h(0)|=\alpha \mu.
	\ena
\end{lemma}

To define iterates of the averaging operator on a function $h(\cdot)$, let $A_{x+1}^0h=h(x)$ and
\beas
A_{x+1}^n=A_{x+1}(A_{\sbull +1}^{n-1}) \qm{for $n \ge 1$,}
\enas
and for a class of functions ${\cal H}$ let 
\beas
A_{x+1}^n({\cal H})=\{A_{x+1}^nh: h \in {\cal H}\}\qm{for $n \ge 0$.}
\enas

\begin{lemma}  \label{lem:alpha.rho^n}
	Let $s(\cdot)$ satisfy Condition \ref{cond:s} and be locally absolutely continuous on $[0, \infty)$. If
	there exists $\rho \in [0,\infty)$ such that 

	 \eqref{eq:18.for.t} holds, then for all $\theta>0, \alpha \ge 0$ and $n \ge 0$, 
	\beas
	A_{x+1}^n ({\rm Lip}_{\alpha,0}) \subset {\rm Lip}_{\alpha \rho^n,0}.
	\enas
\end{lemma}

In the following, by replacing $h(x)$ by $h(x)-E[h(D_{\theta,s})]$, when handling the Stein equations \eqref{eq:f.stein} and \eqref {steineq}, without loss of generality we may assume that $E[h(D_{\theta,s})]=0$.

For a given function $h \in {\rm Lip}_{\alpha,0}$ for some $\alpha \ge 0$, let
\bea \label{def:hstar.g}
h^{(\star k)}(x)=A_{x+1}^k h \qmq{for $k \ge 0$,} g(x)=\sum_{k \ge 0} h^{(\star k)}(x) \qmq{and} g_n(x)=\sum_{k=0}^n h^{(\star k)}(x).
\ena
Also recall definition \eqref{supnorm} that for any $a \ge 0$ and function $f(\cdot)$, $\|f\|_{[0,a]}= \sup_{x \in [0,a]}|f(x)|$.

\begin{theorem}\label{Lipschitzg}
Let $s(\cdot)$ satisfy Condition \ref{cond:s} and be locally absolutely continuous on $[0, \infty)$. Further assume that $\mu=E[D_{\theta,s}]$ exists. If there exists $\rho\in [0,1)$ such that \eqref{eq:18.for.t} holds, then for all $a \ge 0$ and $h \in {\rm Lip}_{1,0}$ we have
\bea \label{eq:hstar.sup.bound}
\|h^{(\star k)}\|_{[0,a]} \le (\mu+a) \rho^k,
\ena
$g_n \in {\rm Lip}_{{(1-\rho^{n+1})/(1-\rho)}}$ and $g(\cdot)$ given by \eqref{def:hstar.g} is a ${\rm Lip}_{{1/(1-\rho)}}$ solution to \eqref{steineq}.
\end{theorem}

\noindent {\em Proof of Theorem \ref{bound}:} The proof follows by arguing as in the proof of Theorem 1.3 of \cite{Go17v1}, with the final claim obtained by applying Theorem \ref{concaverho} to $s(x)=x$; we omit the details.

In the remainder of this section we specialize to the case of the generalized Dickman distribution where for some $\theta>0$ we have $t(x)=x^\theta$, 
$d\nu/dv = \theta v^{\theta-1}$ and
the Stein equation \eqref{eq:f.stein} becomes
\bea \label{eq:f.stein.spec}
(x/\theta) f'(x)+f(x)-f(x+1)=h(x)-E[h(D_{\theta})].
\ena 
Note that the function $s(x)=x$ trivially satisfies Condition \ref{cond:s}. For notational simplicity, in what follows, let $\rho_i=\theta/(\theta+i)$ for $i \in \{1,2\}$.

\begin{lemma}\label{bddsecondderivative}
For non-negative $\alpha$ and $\beta$, let $\mathcal{H}_{\alpha,\beta}$ be as in \eqref{eq:def.calHab}. For every $\theta > 0$, if 
$h \in \mathcal{H}_{\alpha,\beta}$ then $A_x h \in C^2[(0,\infty)]$ and both $A_x h$ and $A_{x+1} h$ are elements of $\mathcal{H}_{\alpha \rho_1,\beta \rho_2}$.
\end{lemma}

\begin{proof} Take $h \in {\cal H}_{\alpha,\beta}$. Since $h \in {\rm Lip}_\alpha$, by Lemmas \ref{lem:alpha.rho^n} and  \ref{lem:t.rho.integral.bound}, $h(\cdot)$ is $\nu$-integrable on any interval of the form $[0,a]$ for all $a > 0$, $A_xh \in {\rm Lip}_{\alpha \rho_1}$ and
\beas
(A_x h)'=\frac{\theta}{x^{\theta+1}} \int_0^x h'(v) v^\theta dv \qmq{for $x>0$.}
\enas
Taking another derivative we obtain
\beas
(A_xh)'' = \frac{\theta}{x^{\theta + 1}} \left[ h'(x)x^{\theta}  - \frac{\theta+1}{x} \int_0^x  h'(v) v^\theta dv \right] \qmq{for $x>0$.}
\enas
As $h' \in {\rm Lip}_\beta$, the function $A_xh$ is twice continuously differentiable on $(0,\infty)$ proving the first claim.
Since
\beas
x^{\theta}=\frac{\theta + 1}{x}\int_0^x v^\theta dv
\enas
we have
\beas
(A_xh)'' = \frac{\theta (\theta + 1)}{x^{\theta + 2}} \left[\int_0^x  (h'(x)  - h'(v)) v^\theta dv \right].
\enas
Taking absolute value and using that $h' \in {\rm Lip}_\beta$ now yields
\begin{multline*}
|(A_xh)''| \le \frac{\theta (\theta + 1)}{x^{\theta + 2}} \left[\int_0^x |h'(x)  - h'(v)|v^\theta  dv \right] \\ \le  \frac{\beta \theta (\theta + 1)}{x^{\theta + 2}} \left[\int_0^x (x-v) v^\theta  dv \right]
= \frac{\beta  \theta (\theta + 1)}{x^{\theta + 2}} \frac{x^{\theta + 2}}{(\theta+1)(\theta + 2)}= \frac{\beta\theta}{\theta +2}=\beta \rho_2. 
\end{multline*}
Since both $A_x h$ and $(A_xh)'$ are continuous at $0$ and belong in $C^1[(0,\infty)]$, we obtain $A_x h \in \mathcal{H}_{\alpha \rho_1,\beta \rho_2}$. The final claim is a consequence of the fact that $A_{x+1}h$ is a left shift of $A_x h$.
\end{proof}

\begin{theorem}\label{fsolutionbds}
For every $\theta > 0$ and $h \in \mathcal{H}_{1,1}$, there exists a solution $f  \in \mathcal{H}_{\theta, \theta/2}$ to \eqref{eq:f.stein.spec} with $\|f'\|_{(0,\infty)} \le \theta$ and $\|f''\|_{(0,\infty)} \le \theta/2$.
\end{theorem}
\begin{proof}
Take $h \in {\cal H}_{1,1}$. By replacing $h(\cdot)$ by $h-E[h(D_\theta)]$ we may assume $E[h(D_\theta)]=0.$ Clearly $s(x)=x$ satisfies Condition \ref{cond:s} and $E[D_\theta]= \theta$ (see e.g. \cite{Dev}). Also, by Theorem \ref{bound}, $\rho=\rho_1$ satisfies \eqref{eq:18.for.t}. For $h \in {\rm Lip}_{1,0}$, Theorem \ref{Lipschitzg} shows that $g(\cdot)$ given by \eqref{def:hstar.g} is a
${\rm Lip}_{1/(1-\rho_1)}$ solution  to \eqref{steineq}. Since $g(\cdot)$ is Lipschitz, we have $g \in \bigcap_{a>0}L^1([0,a],\nu)$ and hence $f(x)=A_x g$ is a solution to \eqref{eq:f.stein.spec} by the equivalence of \eqref{eq:f.stein} and \eqref{steineq}. Now for $a>0$, for any function $h \in L^1([0,a],\nu)$, 
\bea \label{eq:Adot.bound}
\|A_{\sbull} h\|_{[0,a]} = \sup_{x \in [0,a]} |A_{x} h| \le \sup_{x \in [0,a]}\frac{1}{x^\theta} \int_0^{x} |h(v)| \theta v^{\theta-1} dv \le \|h\|_{[0,a]}.
\ena
Let
\beas 
g_n(x)=\sum_{k=0}^n h^{(\star k)}(x) \qmq{and} f_n(x)=A_x g_n.
\enas
Since $g_n \in {\rm Lip}_{{(1-\rho^{n+1})/(1-\rho)}}$ by Theorem \ref{Lipschitzg}, it is $\nu$-integrable over $[0,a]$. Now using \eqref{eq:Adot.bound}, the triangle inequality and  \eqref{eq:hstar.sup.bound} of Theorem \ref{Lipschitzg}, noting $E[D_\theta]=\theta$, we have
\begin{multline*}
\|f - f_n \|_{[0,a]}= \|A_{\sbull} g - A_{\sbull}g_n \|_{[0,a]} \le \| g - g_n \|_{[0,a]} \\ \le \sup_{x \in [0,a]}  \sum_{k \ge n+1} \| h^{(\star k)}\|_{[0,a]} \le (\theta+a)\sum_{k \ge n+1} \rho_1^k = (\theta+a) \frac{\rho_1^{n+1}}{1-\rho_1}.
\end{multline*}
Letting $n \to \infty$, we obtain
\beas
f(x)=\sum_{n \ge 0}A_x h^{(\star n)}.
\enas
Lemma \ref{bddsecondderivative} and induction imply that $A_xh^{(\star n)} \in C^2[(0,\infty)]$ and
\beas
A_xh^{(\star n)} \in \mathcal{H}_{\rho_1^{n+1},\rho_2^{n+1}}
\qmq{for all $n \ge 0$,}
\enas
and hence
\bea\label{derbd}
\|(A_xh^{(\star n)})'\|_{(0,\infty)} \le \rho_1^{n+1} \qmq{and} \|(A_xh^{(\star n)})''\|_{(0,\infty)} \le \rho_2^{n+1}.
\ena
Thus, for any $a>0$, on the interval $(0,a]$, $f'_n(x) = \sum_{k = 0}^n (A_{x} h^{(\star k)})'$ and $f''_n(x)=\sum_{k = 0}^n (A_{x} h^{(\star k)})''$ converge uniformly to the corresponding infinite sums respectively, noting that by \eqref{derbd}, the infinite sums are absolutely summable. Thus we obtain (see e.g. Theorem 7.17 in \cite{Ru64})
\beas
f'(x)=\lim_{n \to \infty}f_n'(x) \qmq{and} f''(x)=\lim_{n \to \infty}f_n''(x) \qm{for all $x \in [0,a]$.}
\enas
Hence, again using \eqref{derbd}, with $\|\cdot\|_{(0,\infty)}$ the supremum norm defined as in \eqref{supnorm},
\beas
\|f'\|_{(0,\infty)} \le \sum_{n \ge 0} \rho_1^{n+1}=\frac{\rho_1}{1-\rho_1}=\theta \qmq{and} \|f''\|_{(0,\infty)} \le \sum_{n \ge 0} \rho_2^{n+1}=\frac{\rho_2}{1-\rho_2}=\frac{\theta}{2}.
\enas
Finally, since $f(\cdot)$ and $f'(\cdot)$ are differentiable everywhere on $(0,\infty)$ with bounded derivative, they are absolutely continuous on $(0,\infty)$. Also both $f(\cdot)$ and $f'(\cdot)$ are continuous at $0$ since by definition, $f(0)=A_0 g=g(0)=\lim_{x \downarrow 0} f(x)$ and $f'(0)=\lim_{x \downarrow 0} f'(x)$. Now noting that if a function is absolutely continuous on $(0,\infty)$ with bounded derivative and continuous at $0$, then it is Lipschitz, we obtain that $f \in {\cal H}_{\theta,\theta/2}$.
\end{proof}

\begin{remark}
The reasoning in the proof of Theorem \ref{fsolutionbds} holds in greater generality in $t(\cdot)$, and only specifically depends on the form $t(x)=x^\theta$ when invoking Lemma \ref{bddsecondderivative}.
\end{remark}

\begin{remark}
In contrast to the bound $\|f''\| \le 2 \|h'\|$ (see e.g. (2.12) of \cite{CGS}) for the solution of Stein equation in the normal case, one cannot uniformly bound the second derivatives of the solutions $f(\cdot)$ of \eqref{eq:f.stein.spec} in Theorem \ref{fsolutionbds} assuming  only a Lipschitz condition on the test functions $h(\cdot)$ in a class ${\cal H}$. For $b>0$ let
	\beas
	h(x) = \left\{
	\begin{array}{cc}
		0 & x \le b \\
		x-b & x >b.
	\end{array}
	\right.
	\enas
Clearly $h \in {\rm Lip}_1$. Taking $\theta=1$ and $s(x)=x$, the function $g(\cdot)$ as in \eqref{def:hstar.g}, with $h(\cdot)$ replaced by ${\bar h}(\cdot)=h(\cdot)-E[h(D)]$ is Lipschitz and solves \eqref{steineq} by Theorem \ref{Lipschitzg}, hence $f(x)=A_xg$ solves \eqref{eq:f.stein.spec}. Arguing as in the proof of Theorem \ref{fsolutionbds} to interchange $A_x$ and the infinite sum, $f(\cdot)$ is given by 
\bea \label{eq:f.counter}
f(x)= \sum_{k \ge 0} A_x[A_{\sbull +1}^k (\bar h)].
\ena

Consider the term $k=0$ in the sum \eqref{eq:f.counter}. Directly, one may verify that	
	\bea\label{firstder}
	A_x h = \left\{
	\begin{array}{cc}
		0 & x \le b \\
		\frac{(x-b)^2}{2x} & x >b.
	\end{array}
	\right.
	\quad 
	(A_x h)' = \left\{
	\begin{array}{cc}
		0 & x \le b \\
		\frac{1}{2}\left(1-(b/x)^2 \right) & x >b.
	\end{array}
	\right.
	\ena
	and 
	\bea\label{secondder}
(A_x h)'' = \left\{
	\begin{array}{cc}
		0 & x \le b \\
		\frac{b^2}{x^3} & x >b.
	\end{array}
	\right.
	\ena
so in particular, 
\bea \label{eq:b.small.f.double.big}
\lim_{x \downarrow b}(A_x \bar h)''=\lim_{x \downarrow b}(A_x h-Eh(D))''=\lim_{x \downarrow b}(A_x h)''=1/b, 
\ena
which is not bounded as $b \downarrow 0$.

From \eqref{firstder} and \eqref{secondder} respectively, we have that $(A_{x+1}\bar h)' \le 1/2$ and $(A_{x+1}\bar h)'' \le b^2/(x+1)^2 \le b^2$ on $(0, \infty)$, and hence $A_{x+1}\bar h \in {\cal H}_{\alpha, \beta}$ with $\alpha=1/2$ and $\beta=b^2$, By Lemma \ref{bddsecondderivative} with $\rho_1=1/2$ and $\rho_2=1/3$, we have $A_{\sbull +1}^k (\bar h) \in {\cal H}_{\alpha/2^{k-1},\beta/ 3^{k-1}}$ for $k \ge 1$. Hence, again by Lemma \ref{bddsecondderivative},
\bea
 A_x[A_{\sbull +1}^k (\bar h)] \in {\cal H}_{\alpha/ 2^k,\beta/3^k} \qm{on $(0,\infty)$ for $k \ge 1$}.
 \ena
 Summing and substituting the vales of $\alpha$ and $\beta$, we obtain
 \bea \label{f:assum}
 \sum_{k \ge 1} A_x[A_{\sbull +1}^k (\bar h)] \in  {\cal H}_{1/2,b^2/2}.
 \ena
From \eqref{eq:f.counter}, \eqref{eq:b.small.f.double.big} and \eqref{f:assum}, we find that $f''(x)$ may be made arbitrarily large on a set of positive measure by choosing $b>0$ sufficiently small.
\end{remark}
\begin{remark}
Shortly after a draft of this manuscript was posted, as a special case of their work on infinitely divisible laws,
Arras and Houdr{\'e} proved smoothness bounds in \cite{AH17} for a solution to the standard Dickman Stein equation of the form
\bea\label{SteinID}
xt(x) - \int_0^1 t(x+u)du = h(x) - E h(D);
\ena 
this equation corresponds to \eqref{eq:f.stein.spec} upon identifying $t(\cdot)$ and $f'(\cdot)$. Lemma 5.2 in \cite{AH17} shows that when $h(\cdot)$ is in the class ${\cal H}=\{h:\|h\|_\infty \le 1, \|h'\|_\infty \le 1, h'(\cdot) \text{ is continuous}\}$ then there exists a solution $t(\cdot)$ to \eqref{SteinID} with $\|t'\|_\infty \le 1$. The proof of Theorem \ref{thm:numgen} requires a uniform bound on $f'(\cdot)$ over $(0,\infty)$ to control the coefficient of $|\mu-1|$ in \eqref{eq:inf.TU}. As no such bound is provided in \cite{AH17}, in the case $\mu=1$ one can argue as for Theorem \ref{thm:numgen} to produce a version of it for the metric induced by ${\cal H}$. As neither class ${\cal H}$ nor ${\cal H}_{1,1}$ in \eqref{eq:def.calHab} contains the other, the first class requiring the test functions to be uniformly bounded, and the second requiring their derivatives to be Lipschitz, the resulting metrics they induce are incomparable.	
\end{remark}

\Addressc
\Addressl

\end{document}